\documentclass[11pt,a4paper]{article}
\usepackage{amsfonts}
\usepackage{amsmath}
\usepackage{amssymb}
\usepackage{graphicx}%
\newtheorem{theorem}{Theorem}

\newtheorem{corollary}[theorem]{Corollary}

\newtheorem{definition}[theorem]{Definition}

\newtheorem{lemma}[theorem]{Lemma}

\newtheorem{proposition}[theorem]{Proposition}
\newtheorem{remark}[theorem]{Remark}

\newenvironment{proof}[1][Proof]{\noindent\textbf{#1.} }{\ \rule{0.5em}{0.5em}}

\begin{document}

\title{A metaplectic perspective of uncertainty principles in the Linear Canonical Transform domain}
\author{Nuno Costa Dias\footnote{Corresponding author; ncdias@meo.pt}\\University of Lisbon (GFM)
\and Maurice de Gosson\footnote{maurice.de.gosson@univie.ac.at}\\University of Vienna, Faculty of Mathematics (NuHAG)
\and Jo\~ao Nuno Prata\footnote{joao.prata@mail.telepac.pt}\\University of Lisbon (GFM)}
\maketitle

\begin{abstract}
We derive Heisenberg uncertainty principles for pairs of Linear Canonical Transforms of a given function, by resorting to the fact that these transforms are just metaplectic operators associated with free symplectic matrices. The results obtained synthesize and generalize previous results found in the literature, because they apply to all signals, in arbitrary dimension and any metaplectic operator (which includes Linear Canonical Transforms as particular cases). Moreover, we also obtain a generalization of the Robertson-Schr\"odinger uncertainty principle for Linear Canonical Transforms. We also propose a new quadratic phase-space distribution, which represents a signal along two intermediate directions in the time-frequency plane. The marginal distributions are always non-negative and permit a simple interpretation in terms of the Radon transform. We also give a geometric interpretation of this quadratic phase-space representation as a Wigner distribution obtained upon Weyl quantization on a non-standard symplectic vector space. Finally, we derive the multidimensional version of the Hardy uncertainty principle for metaplectic operators and the Paley-Wiener theorem for Linear Canonical Transforms.
\end{abstract}



\section{Introduction}

One of the fundamental results in harmonic analysis is the uncertainty principle. In simple terms it states that a function and its Fourier transform cannot be both sharply localized. It was originally stated by the physicist Werner Heisenberg \cite{Heisenberg} in the context of quantum mechanics in terms of some loosely defined concepts of noise and disturbance for an x-ray experiment. \textit{Wave mechanics} was proposed independently by Schr\"odinger. In this context the wave function is interpreted as the \textit{probability amplitude} for the position of a particle. The associated probability amplitude for the momentum of the particle was posited to be the Fourier transform of the wave function. This paved the way of the first mathematical statement of the uncertainty principle, which was proved by Weyl \cite{Weyl}, Kennard \cite{Kennard} and Robertson \cite{Robertson}. For any $L^2 $ function $f$, we have:
\begin{equation}
\left(\int_{\mathbb{R}} (x-a)^2 |f(x)|^2 dx \right) \cdot \left(\int_{\mathbb{R}} (\xi-b)^2 |\widehat{f}( \xi)|^2 d \xi \right) \geq \frac{\|f\|_2^4}{16 \pi^2} ~,
\label{eqIntro1}
\end{equation} 
for all $a,b \in \mathbb{R}$ and where $\widehat{f} $ denotes the Fourier transform of $f$, and $\|f\|_2$ is the $L^2$ norm. If we set
\begin{equation}
a= \langle x \rangle = \frac{\int_{\mathbb{R}} x |f(x)|^2 dx}{\|f\|_2^2}~, ~b= \langle \xi \rangle = \frac{\int_{\mathbb{R}} 
\xi |\widehat{f}( \xi)|^2 d\xi}{\|f\|_2^2}~,
\label{eqIntro2}
\end{equation} 
we obtain from (\ref{eqIntro1}):
\begin{equation}
\Delta x^2 \Delta \xi^2 \geq \frac{\|f\|_2^4}{16 \pi^2} ~,
\label{eqIntro3}
\end{equation} 
where the position and momentum (or time and frequency) spreads are given by:
\begin{equation}
\Delta x^2 = \int_{\mathbb{R}} (x-\langle x \rangle )^2 |f(x)|^2 dx~, ~\Delta \xi^2 = \int_{\mathbb{R}} (\xi-\langle \xi \rangle )^2 |\widehat{f}(\xi)|^2 d\xi~.
\label{eqIntro4}
\end{equation} 
One can think of the Fourier transform as a process by which the time-frequency content of a signal is rotated by $90^o$ from the time to the frequency axis. This interpretation will become clear later. Certain optical systems can rotate a signal by an arbitrary angle. Mathematically this is implemented by the \textit{fractional Fourier transform} (FRFT) \cite{Chen,GoLuFrac,Namias,Ozaktas}.

The FRFT is one of several transforms called \textit{linear canonical transforms} (LCT) which are widely used in optics and signal processing \cite{Bastiaans}. 

A LCT is a singular integral operator of the form
\begin{equation}
\mathcal{L}_S \left[f \right] (\xi)= \widehat{f}_S (\xi)= \frac{1}{\sqrt{2\pi i b}} \int_{\mathbb{R}} f(x) e^{\frac{i\pi}{b}\left(d \xi^2
+ a x^2 -2 x \xi\right)} dx~,
\label{eqIntro5}
\end{equation} 
for $f \in \mathcal{S} (\mathbb{R})$ and $a,c,d \in \mathbb{R}$, $b \in \mathbb{R} \backslash \left\{0 \right\}$.

The subscript $S$ stems from the fact that every LCT is associated to a (free) symplectic matrix
\begin{equation}
S= \left(
\begin{array}{c c}
a &b\\
c & d
\end{array}
\right)~.
\label{eqIntro6}
\end{equation} 
Notice that if we choose $a=d=0$ and $b=-c=1$, we recover (up to multiplicative constant) the Fourier transform. The associated symplectic matrix
\begin{equation}
J=\left(
\begin{array}{c c}
0 &1\\
-1 & 0
\end{array}
\right)~,
\label{eqIntro7}
\end{equation} 
also known as the \textit{standard symplectic matrix}, implements the rotation by $90^o$ in the time-frequency plane that we mentioned before:
\begin{equation}
\left(
\begin{array}{c c}
0 &1\\
-1 & 0
\end{array}
\right)~\left(
\begin{array}{c}
x\\
\xi
\end{array}
\right)=\left(
\begin{array}{c}
\xi\\
-x
\end{array}
\right)~.
\label{eqIntro8}
\end{equation} 
In that respect, the uncertainty principle (\ref{eqIntro3}) can be regarded as one of an infinite number of uncertainty principles expressing the impossibility of simultaneously localizing sharply the representation of a signal along two distinct directions in the time-frequency plane. More specifically, for two distinct symplectic matrices $S^{(1)}$ and $S^{(2)}$ $\left(S^{(1)} \neq \pm S^{(2)} \right)$ associated with the LCTs $\widehat{f}_{S^{(1)}} (x)$ and $\widehat{f}_{S^{(2)}} (\xi)$, respectively, one can anticipate that there exists $C>0$, such that (cf. the definition of $\Delta_{S^{(2)}} \xi^2$ in eq.(\ref{eq9})):
\begin{equation}
\Delta_{S^{(1)}} x^2 \Delta_{S^{(2)}} \xi^2 \geq C \|f\|_2^4~. 
\label{eqIntro9}
\end{equation} 
The constant $C$ can only depend on $S^{(1)}$, $S^{(2)}$ (and on the dimension in the higher-dimensional case). The uncertainty principle (\ref{eqIntro3}) would correspond to the case $S^{(1)}=I$ and $S^{(2)}=J$.

Various uncertainty principles have been proposed in one dimension \cite{Dang,Guanlei1,Guanlei2,Guanlei3,Kou,Ozaktas,Sharma,Yang,Zang,Zhao1,Zhao2} and in two dimensions \cite{Ding}. However, they can be only considered as partial results: they do not address the general $n$ dimensional case, some of them only consider real signals, or only a particular LCT, or else the constant in (\ref{eqIntro9}) is not really a constant as it depends on other quantities such as $\Delta x^2$ or $\Delta \xi^2$.

Here we derive an uncertainty principle of the form (\ref{eqIntro9}) which is valid in arbitrary dimension, for real or complex signals, and such that the constant $C$ depends only on $S^{(1)},~S^{(2)}$ and $n$ (see eq.(\ref{eq31})) and is optimal. 

Moreover, we shall take full advantage of the properties of the metaplectic group. Indeed, LCTs are just metaplectic operators which project onto free symplectic matrices. Using the symplectic covariance property of Weyl operators, we will derive a simple and compact formula for the uncertainty principle (\ref{eqIntro9}) in the most general case. In fact the uncertainty principles obtained are valid not just for LCT but for more general metaplectic transformations as well.

Uncertainty principles of the form (\ref{eqIntro3},\ref{eqIntro9}) do not account for the correlations $x-\xi$. If one performs a linear transformation of the observables $(x, \xi) \mapsto M (x, \xi)$, $M \in \operatorname*{Gl}(2n,\mathbb{R})$, the uncertainty principle is immediately changed. It is well known that one can obtain a stronger version of the uncertainty principle (\ref{eqIntro3}) known as the \textit{Robertson-Schr\"odinger uncertainty principle}, which is invariant under affine linear symplectic transformations. In the present work, we will derive a Robertson-Schr\"odinger version of the uncertainty principle (\ref{eqIntro9}) which is also invariant under a certain kind of linear transformations.

The strategy of our approach resides in the definition of a new kind of time-frequency representation. The Wigner-Ville function $W_{\sigma} f(x, \xi)$ of a signal $f$ is commonly assumed to represent the time-frequency content of the signal. It is not a joint probability density for time and frequency as it may (and in general it does) take on negative values. However it does have the correct marginals\footnote{Provided that $f, \widehat{f} \in L^1 (\mathbb{R}) \cap L^2 (\mathbb{R})$.}
\begin{equation} 
\int_{\mathbb{R}} W_{\sigma} f(x, \xi) d \xi= | f(x)|^2~, ~\int_{\mathbb{R}} W_{\sigma} f(x, \xi) d x= | \widehat{f}(\xi)|^2~,
\label{eqIntro10}
\end{equation} 
and permits the evaluation of the various moments (under suitable conditions on $f$) , 
\begin{equation}
\langle \xi^k \rangle=\int_{\mathbb{R}}\xi^k |\widehat{f} (\xi)|^2 d \xi~, \hspace{0.5 cm}\langle x^k \rangle=\int_{\mathbb{R}}x^k |f (x)|^2 d x~,~k \in \mathbb{N}
\label{eqIntro10A}
\end{equation} 
with the following suggestive formulae\footnote{See also eq.(\ref{eqReview25}).}:
\begin{equation} 
\langle \xi^k \rangle =\int_{\mathbb{R}}\int_{\mathbb{R}} \xi^k W_{\sigma} f(x, \xi) dx d \xi~, ~\langle x^k \rangle =\int_{\mathbb{R}}\int_{\mathbb{R}} x^k W_{\sigma} f(x, \xi) d x d \xi~.
\label{eqIntro11}
\end{equation} 
In the same vein, we will introduce a new distribution, $W_{\vartheta} f(x, \xi)$, for the \textit{rotated} variables $x$, $\xi$ associated with symplectic matrices $S^{(1)}$, $S^{(2)}$, such that
\begin{equation}
\langle x^k \rangle =\int_{\mathbb{R}}\int_{\mathbb{R}} x^k W_{\vartheta} f(x, \xi) d x d\xi~, ~ \langle \xi^k \rangle =~\int_{\mathbb{R}} \int_{\mathbb{R}} \xi^k W_{\vartheta} f(x,\xi) dx d \xi ~.
\end{equation}
Moreover:
\begin{equation} 
\int_{\mathbb{R}} W_{\vartheta} f(x,\xi) d \xi= | \widehat{f}_{S^{(1)}}(x)|^2~, ~\int_{\mathbb{R}} W_{\vartheta} f(x,\xi) d x= | \widehat{f}_{S^{(2)}} (\xi)|^2~.
\label{eqIntro12}
\end{equation} 
The last identities hold because the integrals on the left-hand side coincide with the symplectic Radon transform of the Wigner distribution $W_{\sigma} f(x, \xi)$ \cite{Gosson1}.
The subscripts $\sigma$ and $\vartheta$ will become clear in the sequel. In terms of the distribution $W_{\vartheta} f(x,\xi)$ we also compute the covariance matrix and obtain the uncertainty principle (\ref{eqIntro9}) in a straightforward way.

The reader familiar with time-frequency distributions, should be advised that this distribution is not in the so-called Cohen class \cite{Cohen}. However, in certain cases, it can be regarded as \textit{linear perturbation of the Wigner distribution} in the spirit of \cite{Cordero}. We also show that it emerges naturally as the result of using a process of \textit{Weyl quantization} on a non-standard symplectic vector space. The non-standard symplectic form $\vartheta$ is determined by the symplectic matrices $S^{(1)},~S^{(2)}$. The methods used in this approach were developed by the authors in \cite{Dias1,Dias2,Dias3,Dias4,Dias5} in the context of a deformation of ordinary quantum mechanics known as \textit{noncommutative quantum mechanics}.

Finally, we also consider the generalizations of the Hardy uncertainty principle \cite{Hardy} and the Paley-Wiener theorem \cite{Wiener} to the linear canonical transform domain. Recall that, in loose terms, Hardy's uncertainty principle precludes the possibility of a function $f$ and its Fourier transform $\widehat{f}$ decaying faster than a Gaussian of minimal uncertainty. We prove that a similar result holds for the pair $\widehat{f}_{S^{(1)}}$ and $\widehat{f}_{S^{(2)}}$. 

The Paley-Wiener theorem states that a function $f$ is of compact support if and only if its Fourier transform $\widehat{f}$ extends to an entire function of a certain exponential type. Again, we extend this result to the pair $\widehat{f}_{S^{(1)}}$ and $\widehat{f}_{S^{(2)}}$. 

\vspace{0.3 cm}
Here is a brief summary of our results.

\begin{itemize}
\item We derive the Heisenberg uncertainty principle for any pair of metaplectic operators in Theorem \ref{TheoremHUPLCT} and Corollary \ref{Corollary1}. The novelty of our result are the facts that it is valid in arbitrary dimension, for all metaplectic operators (not just LCT), and that the optimal constant is expressed in terms of the full symplectic matrices that determine the metaplectic operators.

\item We obtain a counterpart of the Robertson-Schr\"odinger uncertainty principle for metaplectic operators which takes into account all elements of the covariance matrix and not just the variances (Theorem \ref{TheoremRobertsonSchrodinger}). To the best of our knowledge, this is a new result in the multidimensional case.

\item We obtain the set of linear transformations that leave the previous Robertson-Schr\"odinger uncertainty principle invariant (Proposition \ref{Proposition1}). In this context they play the same r\^ole as the symplectic transformations in the case of the Fourier transform. Moreover, they are important for quantum optics, because they can be easily implemented experimentally (through mirrors, beam spliters, phase shifters). 

\item We define a new bilinear distribution in the LCT-phase space associated with two arbitrary metaplectic operators (Definition \ref{DefinitionNewBilinearDist}), and give an interpretation in terms of the Weyl quantization on a suitable non-standard symplectic vector space. This new distribution does not belong to Cohen's class (Theorem \ref{TheoremCohenClass}) but it may be a linear perturbation of the Wigner function in some cases (Theorem \ref{TheoremLinearPerturbation}). It permits us to compute the elements of the covariance matrix associated with a pair of metaplectic operators in very simple terms (Lemma \ref{LemmaCovarianceMatrix}). It yields the correct marginals for the two metaplectic operators (Theorem \ref{TheoremMarginal}), it establishes a promising connection with quantum tomography (Remark \ref{RemarkRadon}) and provides the means to define a new infinite family of Cohen classes of bilinear distributions - one Cohen class for each symplectic form (Remark \ref{RemarkNewCohenClass}). 

\item We prove a Hardy uncertainty principle for LCT in one dimension (Theorem \ref{TheoremHardyLCT2}, Corollary \ref{CorollaryHardy1}) and use it to prove the multidimensional version (Theorem \ref{TheoremMultidimensionalHardy}, Corollary \ref{CorollaryMultidimensionalHardyLCT}).

\item We prove a version of the Paley-Wiener theorem for the LCT  (Theorem \ref{TheoremPaleyWiener1}).
\end{itemize}

\section*{Notation}

The transpose of a matrix $A$ is $A^T$. We denote by $X_{j,k}$ the $(j,k)$-th entry of a matrix $X$. An element of the time-frequency plane (sometimes also called phase space) $\mathbb{R}^{2n}= \mathbb{R}_x^n \times \mathbb{R}_{\xi}^n$ is usually denoted by $z=(x, \xi)$, where (depending on the context) $x$ and $\xi$ may be time and frequency of a signal, or position and momentum of some quantum mechanical particle. We will reserve the Greek letters $\xi , \eta \in \mathbb{R}^n$ for the variable after some linear canonical transform. Given a symmetric $n \times n$ matrix $M$, we shall interchangeably use the notations $x \cdot M x$ or $Mx^2$ for the quadratic form $\sum_{j,k=1}^n M_{j,k}x_jx_k$.

We shall use Greek letters $\alpha, \beta = 1, \cdots, 2n$ for time-frequency (phase space) indices and Latin letters $j,k= 1, \cdots, n$ for time or frequency indices.

Given some probability density $\mu$ defined on $\mathbb{R}^n$,
\begin{equation}
\mu(x) \geq 0~, ~ \int_{\mathbb{R}^n} \mu(x) dx = 1~.
\label{eqNotation1}
\end{equation}
we define the expectation value of some function $f(x)$ as
\begin{equation}
\langle f(x) \rangle = \int_{\mathbb{R}^n} f(x) \mu(x) dx~,
\label{eqNotation2}
\end{equation}
whenever the previous integral is absolutely convergent.

Thus, for instance
\begin{equation}
\langle x_j \rangle = \int_{\mathbb{R}^d} x_j \mu(x) dx~,
\label{eqNotation3}
\end{equation}
$j=1,\cdots,n$. We write these components collectively as a vector:
\begin{equation}
\langle x \rangle= \left(\langle x_1 \rangle, \cdots,\langle x_n \rangle\right)~.
\label{eqNotation4}
\end{equation}
Likewise,
\begin{equation}
\left| x- \langle x \rangle \right|^2 = \sum_{j=1}^n \left(x_j-\langle x_j \rangle \right)^2 ~ . 
\label{eqNotation5}
\end{equation}
The \textit{spread or dispersion} of $x$ is:
\begin{equation}
\Delta x^2 = \int_{\mathbb{R}^n} \left| x- \langle x \rangle \right|^2 \mu(x) d x ~.
\label{eqNotation6}
\end{equation}
 
The "hat" on the upper case letter $\widehat{A}$ means that $\widehat{A}$ is an operator, whereas $\widehat{f}$ is the Fourier transform of $f \in L^1 (\mathbb{R})$,
\begin{equation}
\widehat{f} (\xi)=\left(\mathcal{F} f \right) (\xi)= \int_{\mathbb{R}^n} f(x)e^{-2 i \pi \xi\cdot x} dx ~,
\label{eqFourier1}
\end{equation} 
which extends to $L^2 (\mathbb{R}^n)$ by usual density arguments.
 
The inner product in $L^2 (\mathbb{R}^n)$ is 
\begin{equation}
\langle f, g\rangle = \int_{\mathbb{R}^n} f(x) \overline{g(x)} dx~,
\label{eqNotation7}
\end{equation}
and $\|f\|_2 = \sqrt{\langle f, f\rangle }$ is the norm.

The Schwartz space of test functions is denoted by $\mathcal{S}(\mathbb{R}^n)$ and its dual are the tempered distributions $\mathcal{S}^{\prime} (\mathbb{R}^n)$. The duality bracket is $<D,f>$, for $D \in \mathcal{S}^{\prime} (\mathbb{R}^n)$ and $f \in \mathcal{S}(\mathbb{R}^n)$. If $D,f \in L^2 (\mathbb{R}^n)$, then:
\begin{equation}
<D,f> = \int_{\mathbb{R}^n} D(x) f(x) dx = \langle D, \overline{f} \rangle ~.
\label{eqNotation8}
\end{equation}

If there is a constant $C>0$ such that $f(z) \leq C g(z)$ for all $z$, then we write simply $f(z) \lesssim g(z)$.

\section{Preliminaries}

\subsection{Symplectic geometry}\label{SectionSymGeom}

The symplectic form on $\mathbb{R}^{2n} \equiv \mathbb{R}^n \times \mathbb{R}^n$ is the skew-symmetric non-degenerate form defined by:
\begin{equation}
\sigma (z,z^{\prime})= z \cdot J^{-1}  z^{\prime}= \xi \cdot x^{\prime} -x \cdot \xi^{\prime} ~,
\label{eqSymp0}
\end{equation}
where 
\begin{equation}
J= \left(
\begin{array}{c c}
0 & I_n\\
-I_n & 0
\end{array}
\right)
\label{eqSymp0A}
\end{equation}
is the standard symplectic matrix, and $z=(x, \xi),~z^{\prime}=(x^{\prime}, \xi^{\prime}) \in \mathbb{R}^{2n}$. Notice that $J^{-1}=J^T=-J$.

The space $\mathbb{R}^{2n}$ endowed with the symplectic form $\sigma$ is called the \textit{standard symplectic vector space} and is denoted by $(\mathbb{R}^{2n}; \sigma)$.

The symplectic group $\operatorname*{Sp}(2n,\mathbb{R})$ consists of all linear automorphisms $s$ of $\mathbb{R}^{2n}$, such that
\begin{equation}
\sigma \left(s(z),s(z^{\prime}) \right)= \sigma(z,z^{\prime})~,
\label{eqSymp1}
\end{equation}
for all $z,z^{\prime} \in \mathbb{R}^{2n}$.

By abuse of notation, we shall frequently refer to the matrix $S$ representing $s$ in the canonical basis as the symplectic transformation:
\begin{equation}
s(z)=Sz~.
\label{eqSymp2}
\end{equation}
From (\ref{eqSymp1}), the matrix $S$ satisfies:
\begin{equation}
S^T J S =J~.
\label{eqSymp3}
\end{equation}

If we write a matrix $S \in  \operatorname*{Sp}(2n,\mathbb{R})$ in block form
\begin{equation}
S= \left(
\begin{array}{c c}
A & B\\
C & D
\end{array}
\right)~,
\label{eqSymp3.1}
\end{equation}
with $A,B,C,D$ some real $n \times n$ matrices, then we conclude from (\ref{eqSymp3}) that:
\begin{equation}
\begin{array}{l}
A^TC,~B^TD \text{ are symmetric, and } A^TD-C^TB=I_n\\
\\
AB^T,~CD^T \text{ are symmetric, and } AD^T-BC^T=I_n
\end{array}
\label{eqSymp3.2}
\end{equation}
If the matrix $B$ is invertible, then the matrix $S$ is said to be a \textit{free symplectic matrix}. To a free symplectic matrix $S_W$ is associated a generating function which is a quadratic form:
\begin{equation}
W(x,x^{\prime})=\frac{1}{2}x \cdot DB^{-1} x- x \cdot B^{-1} x^{\prime} + \frac{1}{2} x^{\prime} \cdot B^{-1} A x^{\prime}~.
\label{eqSymp3.3}
\end{equation}
This terminology comes from the fact that the knowledge of $W(x,x^{\prime})$ uniquely determines the free symplectic matrix $S$. We have:
\begin{equation}
\left(
\begin{array}{c}
x\\
\xi
\end{array}
\right)=\left(
\begin{array}{c c}
A & B\\
C & D
\end{array}
\right) \left(
\begin{array}{c}
x^{\prime}\\
\xi^{\prime}
\end{array}
\right)\Leftrightarrow \left\{
\begin{array}{l}
\xi= \nabla_x W(x,x^{\prime})\\
\xi^{\prime}=- \nabla_{x^{\prime}} W(x,x^{\prime})
\end{array}
\right.
\label{eqSymp3.4}
\end{equation}
The interest of free symplectic matrices resides in the fact that they generate the symplectic group $\operatorname*{Sp}(2n,\mathbb{R})$. More precisely, every $S \in \operatorname*{Sp}(2n,\mathbb{R})$ can be written as the product $S=S_W S_{W^{\prime}}$ of two free symplectic matrices $S_W$, $S_{W^{\prime}}$.

The standard symplectic form is of prime importance in symplectic geometry. Nevertheless it is not the unique symplectic form. Given a skew-symmetric,  $2n \times 2n$ matrix $\Omega \in \operatorname*{Gl}(2n,\mathbb{R})$, one defines the associated symplectic form on $\mathbb{R}^{2n}$:
\begin{equation}
\vartheta \left(z, z^{\prime} \right)= z \cdot \Omega^{-1} z^{\prime}~, ~ \forall z, z^{\prime} \in \mathbb{R}^{2n}.
\label{eqSymp4}
\end{equation}
A famous theorem in symplectic geometry \cite{Cannas} states that any such symplectic form can be brought to the standard form by a linear transformation\footnote{This is a particular case of Darboux' theorem.}. In other words, there exists $D \in \operatorname*{Gl}(2n,\mathbb{R})$, such that:
\begin{equation}
\vartheta \left(Dz , D z^{\prime} \right)= \sigma (z, z^{\prime})~, ~ \forall z, z^{\prime} \in \mathbb{R}^{2n} \Leftrightarrow \Omega = D J D^T~.
\label{eqSymp5}
\end{equation}
Notice that the matrix $D$ is not unique. In fact, it can be easily shown that, if $D$ and $D^{\prime}$ satisfy (\ref{eqSymp5}), then:
\begin{equation}
D^{-1}D^{\prime}  \in \operatorname*{Sp}(2n,\mathbb{R})~.
\label{eqSymp6}
\end{equation}
We shall call the set of all matrices $D$ which satisfy (\ref{eqSymp5}) the set of \textit{Darboux matrices} and denote it by $\mathcal{D}(2n; \vartheta)$. The \textit{non-standard symplectic vector space} $\mathbb{R}^{2n}$ with symplectic form $\vartheta$ is written $(\mathbb{R}^{2n}; \vartheta)$.

The set of automorphisms of the symplectic vector space $\left(\mathbb{R}^{2n}, \vartheta \right)$ with symplectic form $\vartheta$ is denoted by $\operatorname*{Sp}_{\vartheta}(2n,\mathbb{R})$. Strictly speaking we should write $\operatorname*{Sp}_{\sigma}(2n,\mathbb{R})$ instead of $\operatorname*{Sp}(2n,\mathbb{R})$, but we shall keep the simpler notation, as most part of the present work is concerned with the standard symplectic form $\sigma$. 

Thus, $P \in \operatorname*{Sp}_{\vartheta}(2n,\mathbb{R})$ if
\begin{equation}
\vartheta \left(Pz,Pz^{\prime} \right)= \vartheta \left(z, z^{\prime} \right) ~, ~ \forall z,z^{\prime} \in \mathbb{R}^{2n} \Leftrightarrow \Omega =P \Omega P^T~.
\label{eqSymp7}
\end{equation}
In view of (\ref{eqSymp5}), we conclude that $P \in \operatorname*{Sp}_{\vartheta}(2n,\mathbb{R})$, if and only if there exists a Darboux matrix $D \in \mathcal{D}(2n; )$, such that 
\begin{equation}
D^{-1} PD \in \operatorname*{Sp}(2n,\mathbb{R})~.
\label{eqSymp8}
\end{equation}

As an example of a non-standard symplectic form, we consider the magnetic symplectic form. For details, see \cite{GS}, \S 20. Let $B$ be a $2$-differential form on
$\mathbb{R}^{n}$ (hereafter called magnetic field); one associates to $B$ a
$1$-differential form $A$ (the magnetic potential) defined by $B=dA$. The
forms $A$ and $B$ appear naturally when considering a charged particle moving
in a magnetic field, the Hamiltonian function is then of the type
\[
H(x,p)=\sum_{j=1}^{n}\frac{1}{2m_{j}}(p_{j}-A_{j})^{2}%
\]
where the vector field $(A_{1},...,A_{n})$ is  identified with the
differential form $A=\sum_{j=1}^{n}A_{j}dx_{j}$. We have
\begin{align*}
B  & =dA=\sum_{j=1}^{n}\sum_{k=1}^{n}\frac{\partial A_{k}}{\partial x_{j}%
}dx_{j}\wedge dx_{k}\\
& =\sum_{j<k}\left(  \frac{\partial A_{k}}{\partial x_{j}}-\frac{\partial
A_{j}}{\partial x_{k}}\right)  dx_{j}\wedge dx_{k}.
\end{align*}
Identifying $B$ with the antisymmetric matrix $(B_{jk})_{1\leq j,k\leq n}$,
where
\[
B_{jk}=\frac{\partial A_{k}}{\partial x_{j}}-\frac{\partial A_{j}}{\partial
x_{k}}~,%
\]
we denote by $\sigma_{B}$ the antisymmetric bilinear form defined by
$\sigma_{B}(z,z^{\prime})=z^{T}J_{B}^{-1}(z^{\prime})$
where
\[
J_{B}^{-1}=%
\begin{pmatrix}
B& -I\\
I & 0
\end{pmatrix}
=J^{-1}+%
\begin{pmatrix}
B & 0\\
0 & 0
\end{pmatrix} \Leftrightarrow J_{B}=%
\begin{pmatrix}
0 & I\\
-I & B
\end{pmatrix}
~.
\]
One verifies (\cite{GS}, p.134) that $\sigma_{B}$ is non-degenerate, and hence
a symplectic form. Note that in differential form $\sigma_{B}=\sigma+\pi
^{\ast}B$ where $\pi$ is the projection $\mathbb{R}^{2n}\longrightarrow
\mathbb{R}^{n}$, one verifies that $d\sigma_{B}=d\sigma+\pi^{\ast}dB=0$. 

\subsection{The metaplectic group}

 The group
$\operatorname*{Sp}(2n,\mathbb{R})$ is a connected classical Lie
group; it has covering groups of all orders. Its double cover
admits a faithful (but not irreducible) representation by a group
of unitary operators
on $L^{2}(\mathbb{R}^{n})$, the \textit{metaplectic group} $\operatorname*{Mp}%
(2n,\mathbb{R})$. Thus, to every
$S\in\operatorname*{Sp}(2n,\mathbb{R})$ one
can associate two unitary operators $\widehat{S}$ and $-\widehat{S} \in\operatorname*{Mp}%
(2n,\mathbb{R})$. 

In particular, to every free symplectic matrix one can associate two operators $\widehat{S}_{W,m}$ by:

\begin{equation}
\widehat{S}_{W,m}f(x)=\tfrac{i^{m-n/2}}{\sqrt{|\det B|}}\int%
_{\mathbb{R}^{n}}e^{2 \pi iW(x,x^{\prime})}f(x^{\prime})dx^{\prime} \label{swm}%
\end{equation}
for $f\in\mathcal{S}(\mathbb{R}^{n})$ where $m=0$ mod $2$ if $\det B>0$, and $m=1$ mod $2$ if $\det B<0$. 

It can be shown (Leray \cite{Leray}, de Gosson
\cite{Gosson,wiley}) that these operators extend to unitary operators on $L^2 (\mathbb{R}^n)$ and that every element $\widehat{S}$ of $\operatorname*{Mp} (2n,\mathbb{R})$ can be written (non uniquely) as a product $\widehat{S}_{W,m}\widehat{S}_{W^{\prime},m^{\prime}}$.

It easily follows from the form of the generators (\ref{swm}) of
$\operatorname*{Mp}(2n,\mathbb{R})$ that metaplectic operators are continuous
mappings $\mathcal{S}(\mathbb{R}^{n})\longrightarrow\mathcal{S}(\mathbb{R}%
^{n})$ which extend by duality to continuous mappings
$\mathcal{S}^{\prime
}(\mathbb{R}^{n})\longrightarrow\mathcal{S}^{\prime}(\mathbb{R}^{n})$.
The inverse of the operator $S_{W,m}$ is given by
$S_{W,m}^{-1}=S_{W,m}^{\ast }=S_{W^{\ast},m^{\ast}}$ where
$W^{\ast}(x,x^{\prime})=-W(x^{\prime},x)$ and $m^{\ast}=n-m$.

\subsection{Weyl quantization on $(\mathbb{R}^{2n}; \sigma)$}

The \textit{symplectic Fourier transform} of a function $F \in L^1 (\mathbb{R}^{2n}) \cap L^2 (\mathbb{R}^{2n})$ is given by:
\begin{equation}
(\mathcal{F}_{\sigma} F) (\zeta) =  \int_{\mathbb{R}^{2n}} F(z) e^{- 2i \pi  \sigma (\zeta,z)} dz.
\label{eqReview9}
\end{equation}
It is related to the Fourier transform by:
\begin{equation}
(\mathcal{F}_{\sigma} F)(\zeta) = \widehat{F} (J \zeta)~.
\label{eqReview10}
\end{equation}
The symplectic Fourier transform is an involution which extends by duality to an involutive automorphism $\mathcal{S}^{\prime} (\mathbb{R}^{2n}) \to \mathcal{S}^{\prime} (\mathbb{R}^{2n})$.

Given a symbol $a^{\sigma} \in \mathcal{S}^{\prime} (\mathbb{R}^{2n})$, the associated \textit{Weyl operator} is given by the Bochner integral \cite{Birk,Gosson}:
\begin{equation}
\widehat{A}:=  \int_{\mathbb{R}^{2n}} (\mathcal{F}_{\sigma} a^{\sigma}) (z_0) \widehat{T}^{\sigma} (z_0 ) dz_0,
\label{eqReview11}
\end{equation}
\begin{equation}
(\widehat{T}^{\sigma} (z_0) f) (x)= e ^{ 2 i \pi \xi_0 \cdot\left(x- \frac{x_0}{2} \right)} f(x-x_0),
\label{eqReview12}
\end{equation}
for $z_0 = (x_0, \xi_0) \in \mathbb{R}^{2n}$ and $f \in \mathcal{S} (\mathbb{R}^n)$. We remark that the operator $\widehat{A}$ is formally self-adjoint if and only its symbol $a^{\sigma}$ is real.

The fundamental operators in Weyl quantization are the operators:
\begin{equation}
\left\{
\begin{array}{l l}
\left(\widehat{X}_l f\right)(x)=  x_l f(x),  & l=1, \cdots,n\\
& \\
\left(\widehat{P}_l f \right)(x)= \frac{1}{2 \pi i} \partial_l f(x),  & l=1, \cdots,n
\end{array}
\right.
\label{eq21}
\end{equation}
In quantum mechanics $\widehat{X}_l$ is interpreted as the $l$-th component of the position of a particle and $\widehat{P}_l$ as the $l$-th component of its momentum. We can write them collectively as $\widehat{Z}= \left(\widehat{X}, \widehat{P}\right)$, with $\widehat{Z}_l=\widehat{X}_l$, $\widehat{Z}_{n+l}=\widehat{P}_l$, $l=1, \cdots, n$.

They satisfy the following commutation relations, called the \textit{Heisenberg algebra}:
\begin{equation}
\left[\widehat{Z}_{\alpha},\widehat{Z}_{\beta} \right]= \frac{i}{2 \pi} J_{\alpha, \beta} \widehat{I}~, \hspace{0.5 cm} \alpha, \beta= 1, \cdots, 2n~,
\label{eq21A}
\end{equation}
where $\left(J_{\alpha, \beta} \right)_{\alpha, \beta}$ are the entries of the standard symplectic matrix $J$ and  $\widehat{I}$ is the identity operator.

If $S \in \operatorname*{Sp}(2n,\mathbb{R})$ and $\widehat{S}$ is any of the two metaplectic operators that project onto $S$, then we have:
\begin{equation}
\widehat{S}^{\ast} \widehat{Z}_{\alpha} \widehat{S}= \sum_{\beta=1}^{2n} S_{\alpha, \beta} 
\widehat{Z}_{\beta}~, 
\label{eqCovariance}
\end{equation}
for all $\alpha=1, \cdots, 2n$.

In terms of these operators, one can write:
\begin{equation}
\widehat{T}^{\sigma} (z_0)= e^{2i \pi \sigma (z_0, \widehat{Z})}~.
\label{eq21B}
\end{equation}
The exponentiation of (\ref{eq21A}) leads to the following commutation relations for the Heisenberg-Weyl displacement operators:
\begin{equation}
\widehat{T}^{\sigma} (z_1) \widehat{T}^{\sigma} (z_2)=e^{i \pi \sigma (z_1,z_2)} \widehat{T}^{\sigma} (z_1+z_2) = e^{2 i \pi \sigma (z_1,z_2)}\widehat{T}^{\sigma} (z_2) \widehat{T}^{\sigma} (z_1)~,
\label{eq21C}
\end{equation}
for all $z_1,z_2 \in \mathbb{R}^{2n}$.

The \textit{Weyl correspondence}, written $a^{\sigma} \overset{\mathrm{Weyl}}{\longleftrightarrow} \widehat{A}$ or $\widehat{A} \overset{\mathrm{Weyl}}{\longleftrightarrow} a^{\sigma}$, between an element $a^{\sigma} \in \mathcal{S}^{\prime} (\mathbb{R}^{2n})$ and the Weyl operator it defines is bijective; in fact the Weyl transform is one-to-one from $\mathcal{S}^{\prime} (\mathbb{R}^{2n})$ onto the space $\mathcal{L}\left( \mathcal{S}(\mathbb{R}^{n}),\mathcal{S}^{\prime} (\mathbb{R}^{n})\right)$ of linear continuous maps $\mathcal{S}(\mathbb{R}^{n}) \to \mathcal{S}^{\prime} (\mathbb{R}^{n})$ (see e.g. Maillard \cite{Maillard}, Unterberger \cite{Unterberger} or Wong \cite{Wong}). This can be proven using Schwartz's kernel theorem and the fact that the Weyl symbol $a^{\sigma}$ of the operator $\widehat{A}$ is related to the distributional kernel $K_A$ of that operator by the partial Fourier transform with respect to the y variable
\begin{equation}
a^{\sigma}(x, \xi) = \int_{\mathbb{R}^n} K_A \left( x+ \frac{y}{2},x- \frac{y}{2} \right) e^{- 2 i \pi \xi \cdot y} dy,
\label{eqReview13}
\end{equation}
where $K_A \in  \mathcal{S}^{\prime} (\mathbb{R}^n \times \mathbb{R}^n )$ and the Fourier transform is defined in the usual distributional sense. Conversely, the kernel $K_A$ is expressed in terms of the symbol $a^{\sigma}$ by the inverse Fourier transform
\begin{equation}
K_A(x, y) =  \int_{\mathbb{R}^n} a^{\sigma} \left(\frac{x+y}{2},\xi \right) e^{ 2 i \pi  \xi \cdot (x-y)} d \xi.
\label{eqReview14}
\end{equation}

Weyl operators enjoy the following symplectic covariance property \cite{Folland1,Birk,Gosson,Gro,Wong}. Let $S \in Sp(2n, \mathbb{R})$ and $\widehat{S} \in Mp(2n, \mathbb{R})$ be one of the two metaplectic operators that project onto $S$. If $\widehat{A}: \mathcal{S} (\mathbb{R}^n) \to \mathcal{S}^{\prime} (\mathbb{R}^n)$ is  a Weyl operator with symbol $a^{\sigma} \in  \mathcal{S}^{\prime} (\mathbb{R}^{2n})$, then we have
\begin{equation}
\widehat{S}^{-1} \widehat{A} \widehat{S} \overset{\mathrm{Weyl}}{\longleftrightarrow} a^{\sigma} \circ S .
\label{eqReview14.1}
\end{equation}

An important case consists of rank one operators of the form:
\begin{equation}
\left(\widehat{\rho}_{f,g} h \right) (x) = \langle h , g \rangle f (x),
\label{eqReview17}
\end{equation}
for fixed $f,g \in L^2 (\mathbb{R}^n)$ acting on $h \in L^2 (\mathbb{R}^n)$. They are Hilbert-Schmidt operators with kernel $K_{f,g}(x,y) = (f \otimes \overline{g}) (x,y)=f(x)  \overline{g(y)} $. According to (\ref{eqReview13}), the associated Weyl symbol is:
\begin{equation}
W_{\sigma}\left(f, g \right)(x,\xi) = \int_{\mathbb{R}^n} f \left( x + \frac{y}{2} \right) \overline{g \left( x - \frac{y}{2} \right)} e^{- 2 i \pi \xi \cdot y} dy.
\label{eqReview18}
\end{equation}
This is known as the \textit{cross-Wigner function}.

From (\ref{eqReview14.1}), we conclude that
\begin{equation}
W_{\sigma}(\widehat{S}f,\widehat{S}g)(z)=W_{\sigma}(f,g)(S^{-1} z).
\label{eqReview19.1}
\end{equation}
If $g=f$, we simply write $W_{\sigma}f$ meaning $W_{\sigma}(f,f)$:
\begin{equation}
W_{\sigma}f(x, \xi)=  \int_{\mathbb{R}^n} f \left( x + \frac{y}{2} \right) \overline{f \left( x - \frac{y}{2} \right)} e^{- 2 i \pi \xi \cdot y} dy.
\label{eqReview20}
\end{equation}
One aspect which makes the Wigner formalism very appealing is the fact that expectation values are computed with a formula akin to classical statistical mechanics \cite{Folland,Gosson,Wong}. Indeed, if $\widehat{A}$ is a self-adjoint Weyl operator with symbol $a^{\sigma} \in \mathcal{S} (\mathbb{R}^{2n})$, then it can be shown that
\begin{equation}
\langle \widehat{A} f ,g \rangle= \langle a^{\sigma},  W_{\sigma}(g,f)  \rangle,
\label{eqReview24}
\end{equation}
for $f,g \in  \mathcal{S} (\mathbb{R}^n)$. In particular, we have:
\begin{equation}
<\widehat{A} >_f=\langle f, \widehat{A} f \rangle = \int_{\mathbb{R}^{2n}} a^{\sigma}(x,\xi) W_{\sigma} f(x,\xi) dx d \xi~.
\label{eqReview25}
\end{equation}
Notice that equations (\ref{eqIntro11},\ref{eqIntro12}) are just particular instances of this formula.

We can also extend (\ref{eqReview24}) to the case $a^{\sigma} \in \mathcal{S}^{\prime} (\mathbb{R}^{2n})$ (see eq.(\ref{eqNotation8})):
\begin{equation}
< \widehat{A} f , \overline{g} >= < a^{\sigma},  W_{\sigma}(f,g)  >,
\label{eqReview25A}
\end{equation}
for $f,g \in  \mathcal{S} (\mathbb{R}^n)$, and where we used the fact that $\overline{W_{\sigma}(g,f)}=W_{\sigma}(f,g)$.

Before we complete this section, let us briefly mention the covariance properties of Wigner distributions under phase-space translations. It is a well known fact that \cite{Gosson}:
\begin{equation}
W_{\sigma} \left(\widehat{T}^{\sigma} (z_0) f \right) (z)= W_{\sigma} f (z-z_0)~,
\label{eqReview25A}
\end{equation}
for any $z_0 \in \mathbb{R}^{2n}$.  

\subsection{Weyl quantization on $(\mathbb{R}^{2n}; \vartheta)$}\label{SectionWeylNonStandard}

Alternatively, one may choose to quantize on a non-standard symplectic vector space  $(\mathbb{R}^{2n}; \vartheta)$ \cite{Dias1,Dias2,Dias3,Dias4,Dias5}.

Given a Darboux matrix $D \in \mathcal{D}(2n;\vartheta)$, we define the new operators:
\begin{equation}
\widehat{\Xi}= D \widehat{Z} ~,
\label{eqNonStandQuant1}
\end{equation}
which satisfy the commutation relations:
\begin{equation}
\left[\widehat{\Xi}_{\alpha},\widehat{\Xi}_{\beta} \right]= \frac{i}{2 \pi} \Omega_{\alpha, \beta} \widehat{I}~, \hspace{0.5 cm} \alpha, \beta= 1, \cdots, 2n ~,
\label{eqNonStandQuant2}
\end{equation}
where $\left(\Omega_{\alpha, \beta} \right)_{\alpha, \beta}$ are the entries of the matrix $\Omega$.

Upon exponentiation, we obtain the nonstandard Heisenberg-Weyl displacement operators:
\begin{equation}
\widehat{T}^{\vartheta} (z_0)= e^{2 i \pi \vartheta (z_0, \widehat{\Xi})}~.
\label{eqNonStandQuant3}
\end{equation}
Notice that in view of (\ref{eqNonStandQuant1}), we may write:
\begin{equation}
\widehat{T}^{\vartheta} (z_0)=e^{2 i \pi \vartheta (z_0, \widehat{\Xi})}=e^{2 i \pi \vartheta (z_0, D\widehat{Z})}=e^{2 i \pi \sigma (D^{-1}z_0, \widehat{Z})}=\widehat{T}^{\sigma} (D^{-1} z_0)~,
\label{eqNonStandQuant4}
\end{equation}
and it follows that:
\begin{equation}
\widehat{T}^{\vartheta} (z_1)\widehat{T}^{\vartheta} (z_2)= e^{i \pi \vartheta(z_1,z_2)} \widehat{T}^{\vartheta} (z_1+z_2)= e^{2 i \pi \vartheta(z_1,z_2)} \widehat{T}^{\vartheta} (z_2) \widehat{T}^{\vartheta} (z_1)~,
\label{eqNonStandQuant5}
\end{equation}
for all $z_1,z_2 \in \mathbb{R}^{2n}$.

Also, if we perform the substitution $z_0=D^{-1} z$ in (\ref{eqReview11}), we obtain:
\begin{equation}
\begin{array}{c}
\widehat{A}= \int_{\mathbb{R}^{2n}} \left(\mathcal{F}_{\sigma} a^{\sigma} \right) (z_0)\widehat{T}^{\sigma} (z_0) d z_0= \\
\\
=\frac{1}{\sqrt{\det \Omega}}\int_{\mathbb{R}^{2n}} \left(\mathcal{F}_{\sigma} a^{\sigma} \right) (D^{-1} z)\widehat{T}^{\vartheta} (z) d z~,
\end{array}
\label{eqNonStandQuant6}
\end{equation}
where we used the fact that $\det \Omega = \left(\det D \right)^2$.

Next notice that:
\begin{equation}
\begin{array}{c}
\left(\mathcal{F}_{\sigma} a^{\sigma} \right) (D^{-1} z)=\int_{\mathbb{R}^{2n}} a^{\sigma}(z_0) e^{- 2 i \pi  \sigma (D^{-1} z, z_0)} d z_0=\\
\\
=  \int_{\mathbb{R}^{2n}} a^{\sigma}(z_0) e^{- 2 i \pi  \vartheta (z, D z_0)} d z_0= \\
\\
=\frac{1}{\sqrt{\det \Omega}} \int_{\mathbb{R}^{2n}} a^{\sigma}(D^{-1} z_1) e^{- 2 i \pi  \vartheta (z, z_1)} d z_1~.
\end{array}
\label{eqNonStandQuant7}
\end{equation}
This suggests the definition of the $\vartheta$-symplectic Fourier transform \cite{Dias4,Dias5}:
\begin{equation}
\left(\mathcal{F}_{\vartheta} a \right) (\zeta):= \frac{1}{\sqrt{\det \Omega}} \int_{\mathbb{R}^{2n}} a(z) e^{-2 i \pi \vartheta (\zeta, z)} dz ~, 
\label{eqNonStandQuant8}
\end{equation}
in a distributional sense for $a \in \mathcal{S}^{\prime} (\mathbb{R}^{2n})$, and the $\vartheta$-Weyl symbol:
\begin{equation}
a^{\vartheta} (z):= a^{\sigma} (D^{-1} z)~.
\label{eqNonStandQuant9}
\end{equation}
Altogether from (\ref{eqNonStandQuant6}-\ref{eqNonStandQuant9}), we obtain:
\begin{equation}
\widehat{A}=\frac{1}{\sqrt{\det \Omega}} \int_{\mathbb{R}^{2n}} \left(\mathcal{F}_{\vartheta} a^{\vartheta} \right) (z) \widehat{T}^{\vartheta} (z) d z ~.
\label{eqNonStandQuant10}
\end{equation}
This defines the correspondence principle between an operator $\widehat{A}$ and its $\vartheta$-Weyl symbol $a^{\vartheta}$. 

We thus have for $a^{\sigma} (z) \in \mathcal{S}^{\prime} (\mathbb{R}^{2n})$ and $f, g \in \mathcal{S}(\mathbb{R}^n)$ (see eq.(\ref{eqReview25A})):
\begin{equation}
\begin{array}{c}
< \widehat{A} f,\overline{g}> = <a^{\sigma} ,W_{\sigma} (f,g)> = \\
\\
= \int_{\mathbb{R}^{2n}} a^{\sigma} (z) W_{\sigma}(f,g) (z) dz =\\
\\
= \frac{1}{\sqrt{\det \Omega}} \int_{\mathbb{R}^{2n}}  a^{\sigma} (D^{-1}z_0) W_{\sigma}(f,g) (D^{-1} z_0) dz_0=\\
\\
=\int_{\mathbb{R}^{2n}}  a^{\vartheta} (z) W_{\vartheta}(f,g) (z) dz= <a^{\vartheta} ,W_{\vartheta} (f,g)> ~,
\end{array}
\label{eqNonStandQuant11}
\end{equation}
where
\begin{equation}
W_{\vartheta} (f,g) (z) = \frac{1}{\sqrt{\det \Omega}} W_{\sigma} (f,g) (D^{-1}z) 
\label{eqNonStandQuant12}
\end{equation}
is the $\vartheta$-\textit{cross Wigner distribution} \cite{Dias4,Dias5}.  

In particular the $\vartheta$-Wigner function is:
\begin{equation}
W_{\vartheta} f (z): = W_{\vartheta} (f,f) (z) = \frac{1}{\sqrt{\det \Omega}} W_{\sigma} f (D^{-1}z)~. 
\label{eqNonStandQuant13}
\end{equation}

Regarding the covariance property with respect to phase-space translations, the counterpart of (\ref{eqReview25A}) is:
\begin{equation}
W_{\vartheta} \left(\widehat{T}^{\vartheta} (z_0) f \right) (z)= W_{\vartheta} f (z-z_0)~,
\label{eqReview25B}
\end{equation}
for any $z_0 \in \mathbb{R}^{2n}$.  

Indeed, from (\ref{eqReview25A},\ref{eqNonStandQuant4},\ref{eqNonStandQuant13}):
\begin{equation}
\begin{array}{c}
W_{\vartheta} \left(\widehat{T}^{\vartheta} (z_0) f \right) (z)=\frac{1}{\sqrt{\det \Omega}} W_{\sigma} \left(\widehat{T}^{\sigma} (D^{-1}z_0) f \right) (D^{-1}z)=\\
\\
=\frac{1}{\sqrt{\det \Omega}} W_{\sigma}f \left(D^{-1}z-D^{-1}z_0 \right)=\frac{1}{\sqrt{\det \Omega}} W_{\sigma}f \left(D^{-1}(z-z_0) \right)=\\
\\
=W_{\vartheta} f (z-z_0)~.
\end{array}
\label{eqReview25C}
\end{equation}

\subsection{Linear canonical transform}

Given a free symplectic matrix $S \in \operatorname*{Sp}(2n,\mathbb{R})$, the associated linear canonical transform (LCT) is defined by:
\begin{equation}
\begin{array}{c}
f \in  \mathcal{S}(\mathbb{R}^{n})
\mapsto \mathcal{L}_{S} \left[f \right] (\xi) =\widehat{f}_{S} (\xi) =\\
\\
= \frac{1}{i^{n/2}\sqrt{ \det (B)}} \int_{\mathbb{R}^n} f(x) e^{i \pi \left(\xi \cdot D  B^{-1} \xi +  x \cdot B^{-1} Ax \right) -2 \pi i x \cdot  B^{-1} \xi} d x,
\end{array}
\label{eq1}
\end{equation}
where the symplectic matrix is written in block form (\ref{eqSymp3.1}) satisfying (\ref{eqSymp3.2}).
In view of (\ref{eq1}) we must also require that $\det (B) \neq 0$. In other words, $S$ is a free symplectic matrix. Therefore, we should write it as $S_W$, where $W$ is the corresponding generating function (\ref{eqSymp3.3}).

By comparison with the metaplectic representation (cf.(\ref{swm})), we conclude that:
\begin{equation}
\left( \widehat{S}_{W,m} f\right) (\xi) = \mathcal{L}_{S_W} \left[f \right] (\xi)~.
\label{eqMetaplectic2}
\end{equation}
There is admittedly an ambiguity in the choice of the value of $m$ in (\ref{eqMetaplectic2}). But as this amounts only to a sign, all the relevant results of the present work are independent of it, and we will therefore ignore it.

If $S^{(1)}$, $S^{(2)}$ are two symplectic matrices, such that $\det (B_2 A_1^T-A_2 B_1^T) \neq 0$  (see the block form (\ref{eqSymp3.1})), then:
\begin{equation}
\mathcal{L}_{S^{(1)}S^{(2)}}=\mathcal{L}_{S^{(1)}}\mathcal{L}_{S^{(2)}}~. 
\label{swmextra}
\end{equation}

Notice also that we can writte the LCT transform (\ref{eq1}) as:
\begin{equation}
\widehat{f}_S(\xi)=\mathcal{L}_Sf (\xi)= \frac{1}{\sqrt{| \det B|}} \gamma_1 (\xi) \mathcal{F}\left[\gamma_2 f \right] (B^{-1} \xi )~,
\label{eqExtraLCT1}
\end{equation}
where, as before, $\mathcal{F}$ is the Fourier transform, and $\gamma_1,\gamma_2$ are the two chirp functions: 
\begin{equation}
\gamma_1(\xi)= e^{i \pi \xi \cdot DB^{-1} \xi}~,  
\hspace{0.5 cm} \gamma_2(x)=e^{i \pi x \cdot B^{-1}A x }~.
\label{eqExtraLCT2}
\end{equation}
In view of (\ref{eqExtraLCT1},\ref{eqExtraLCT2}), the LCT  satisfies:
\begin{equation}
\|\widehat{f}_{S} \|_2 = \|f\|_2,
\label{eqMetaplectic4}
\end{equation}
and thus extends to a unitary transformation on $L^2 (\mathbb{R}^n)$.

Particular cases include:

\begin{enumerate}
\item {\bf The Fourier transform:} If we choose $A=D=0$, $B=-C=I_n$, or equivalently
\begin{equation}
S=J= \left(
\begin{array}{c c}
0 & I_n\\
-I_n & 0
\end{array}
\right),
\label{eq4}
\end{equation}
where $J$ is the standard symplectic matrix, then we obtain up to a multiplicative constant the Fourier transform:
\begin{equation}
\mathcal{L}_{J} \left[f \right]  (\xi) = i^{-n/2}\widehat{f} (\xi)= i^{-n/2} \int_{\mathbb{R}^n} f(x) e^{-2 i \pi x \cdot \xi} dx~.
\label{eq5}
\end{equation}

\vspace{0.3 cm}
\item {\bf The fractional Fourier transform:} If we choose $A=D= \text{diag} \left(\cos \theta_1, \cdots ,\cos \theta_n  \right)$ and $B=-C= \text{diag} \left(\sin \theta_1, \cdots ,\sin \theta_n  \right)$, where $\theta_j \in \left(0, 2 \pi \right) \backslash \left\{\pi \right\}$, then we obtain the fractional Fourier transform:
\begin{equation}
\mathcal{L}_{S_{{\bf \theta}}} \left[f \right]  (\xi) = \frac{1}{i^{n/2}\sqrt{\prod_{j=1}^n \sin \theta_j}} \int_{\mathbb{R}^n} f(x) e^{i \pi \sum_{k=1}^n \sin^{-1} \theta_k \left[\left(x_k^2 + \xi_k^2 \right) \cos \theta_k -2 \theta_k x_k \xi_k\right]} dx~.
\label{eq6}
\end{equation}

\vspace{0.3 cm}
\item {\bf The Fresnel transform:} If we choose $A=D= I_n$, $B= \text{diag} (b_1, \cdots,b_n)$ where $b_j \neq 0$,  and $C=0$, then we obtain the Fresnel transform:
\begin{equation}
\mathcal{L}_{S_{{\bf b}}} \left[f \right]  (\xi) = \frac{1}{i^{n/2}\sqrt{\prod_{j=1}^n b_j}} \int_{\mathbb{R}^n} f(x) e^{i \pi \sum_{k=1}^n b_k (x_k - \xi_k)^2} dx~.
\label{eq7}
\end{equation}

\vspace{0.3 cm}
\item {\bf The Lorentz transform:} If we choose $A=D= \text{diag} \left(\cosh \phi_1, \cdots ,\cosh \phi_n  \right)$ and $B=C= \text{diag} \left(\sinh \phi_1, \cdots ,\sinh \phi_n  \right)$, where $\phi_j \neq 0$, then we obtain the Lorentz transform:
\begin{equation}
\mathcal{L}_{S_{{\bf \phi}}} \left[f \right]  (\xi) = \frac{1}{i^{n/2}\sqrt{\prod_{j=1}^n \sinh \phi_j}} \int_{\mathbb{R}^n} f(x) e^{i \pi \sum_{k=1}^n \sinh^{-1} \phi_k \left[\left(x_k^2 + \xi_k^2 \right) \cosh \phi_k -2  \phi_k x_k \xi_k\right]} dx~.
\label{eq8}
\end{equation}

\end{enumerate}

\section{Uncertainty principles}

If $f, |x|f \in L^2 (\mathbb{R}^n)$, then we define the spread in the time domain:
\begin{equation}
\Delta x^2 = \int_{\mathbb{R}^n} |x- <x> |^2 |f(x)|^2dx,
\label{eq9}
\end{equation}
where $<x>$ is as in (\ref{eqNotation4}) and $<x_j>$  is the expectation value:
\begin{equation}
<x_j> = \int_{\mathbb{R}^n} x_j  \frac{|f(x)|^2}{\|f\|_2^2} dx, \hspace{0.5 cm} j=1, \cdots,n~.
\label{eq10}
\end{equation}
Likewise, if $\widehat{f}, |\xi|\widehat{f} \in L^2 (\mathbb{R}^n)$, then we define the spread in the frequency domain:
\begin{equation}
\Delta \xi^2 = \int_{\mathbb{R}^n} |\xi- < \xi> |^2 |\widehat{f}(\xi)|^2 d \xi,
\label{eq9}
\end{equation}
where $<\xi>=(<\xi_1>, \cdots, <\xi_n>)$ is the expectation value:
\begin{equation}
<\xi_j> = \int_{\mathbb{R}^n} \xi_j  \frac{|\widehat{f}(\xi)|^2}{\|f\|_2^2} d \xi, \hspace{0.5 cm} j=1, \cdots,n.
\label{eq10}
\end{equation}
By the same token, given a symplectic matrix $S$ and the associated metaplectic operators $\pm \widehat{S}$, suppose that if $(\widehat{S}f)(\xi) , |\xi|(\widehat{S}f) (\xi) \in L^2 (\mathbb{R}^n)$, then we define the spread in the LCT domain:
\begin{equation}
\Delta_S \xi^2 = \int_{\mathbb{R}^n} |\xi- < \xi> |^2 |(\widehat{S} f)(\xi)|^2 d \xi,
\label{eq9}
\end{equation}
where $<\xi>=(<\xi_1>, \cdots, <\xi_n>)$ is the expectation value:
\begin{equation}
<\xi_j> = \int_{\mathbb{R}^n} \xi_j  \frac{|(\widehat{S}f)(\xi)|^2}{\|f\|_2^2} d \xi, \hspace{0.5 cm} j=1, \cdots,n.
\label{eq10}
\end{equation}
The Heisenberg uncertainty principle in its classic form can be easily derived from the following proposition (for a proof see e.g. Proposition 2.1 \cite{Folland}).
\begin{proposition}\label{PropositionHeisenberg}
Let $\widehat{A}$ and $\widehat{B}$ be essentially self-adjoint operators on a Hilbert space $\mathcal{H}$ and $a,b \in \mathbb{C}$. Then:
\begin{equation}
\|(\widehat{A}-a)f \|_{\mathcal{H}} ~\|(\widehat{B}-b)f \|_{\mathcal{H}} \geq \frac{1}{2} \left| \langle \left[\widehat{A},\widehat{B} \right]f,f \rangle \right|,
\label{eq11}
\end{equation}
for all $f \in \text{Dom} \left(\left[ \widehat{A}, \widehat{B} \right]\right)$. Moreover, we have an equality if and only if
\begin{equation}
\left(\widehat{A} -a \right) f= i c \left(\widehat{B}-b \right)f,
\label{eq12}
\end{equation}
for some $c \in \mathbb{R}$.
\end{proposition}
If we set $\left(\widehat{A}f \right)(x)=\left(\widehat{X}_jf \right)(x)=x_j f(x)$ and $\left(\widehat{B}f \right)(x)=\left(\widehat{P}_j f \right)(x)= \frac{1}{2 \pi i} \partial_j f(x)$ on the intersection of their maximal domains for fixed $j=1, \cdots,n$, then we recover the Heisenberg uncertainty principle:
\begin{equation}
\int_{\mathbb{R}^n} (x_j - a_j)^2 |f(x)|^2 dx ~ \int_{\mathbb{R}^n} (\xi_j - b_j)^2 |\widehat{f}( \xi)|^2 d \xi \geq \frac{\|f\|_2^4}{16 \pi^2},
\label{eq13}
\end{equation}
for $ j=1, \cdots, n$, and $a=(a_1, \cdots,a_n), ~b = (b_1, \cdots, b_n) \in \mathbb{R}^n$. We have an equality in (\ref{eq13}) for all $j=1, \cdots, n$, if and only if $f$ is a generalized Gaussian of the form:
\begin{equation}
f(x)=Ce^{ix \cdot b} e^{-(x-a) \cdot K (x-a)},
\label{eq14}
\end{equation}
where $C \in \mathbb{C}$ and $K= \text{diag} (k_1, \cdots, k_n)$, with $k_j >0$, $j=1, \cdots, n$.

If we set $a=<x>$, $b=<\xi>$, take the square root of both sides of (\ref{eq13}), sum over $j$ and use the Cauchy-Schwarz inequality for vectors in $\mathbb{C}^n$, we arrive at the following $n$-dimensional version of the Heisenberg uncertainty principle:
\begin{corollary}\label{CorollaryHeisenberg}
Let $f \in L^2 (\mathbb{R}^n)$. Then we have:
\begin{equation}
\Delta x^2 ~\Delta \xi^2 \geq  \frac{n^2 }{16 \pi^2} \|f\|_2^4.
\label{eq15}
\end{equation}
Moreover, we have an equality in (\ref{eq15}), if and only if,
\begin{equation}
f(x)=Ce^{i x \cdot < \xi>} e^{-K |x-<x>|^2},
\label{eq16}
\end{equation}
where $C \in \mathbb{C}$ and $K>0$.
\end{corollary}

If we take (\ref{eq9}) into account, we realize that we can write (\ref{eq15}) in the following unusual form:
\begin{equation}
\Delta_{I_n} x^2 ~\Delta_J \xi^2 \geq  \frac{n^2 }{16 \pi^2} \|f\|_2^4.
\label{eq17}
\end{equation}

It thus seems natural to look for uncertainty principles of the form:
\begin{equation}
\Delta_{S^{(1)}} x^2 ~\Delta_{S^{(2)}} \xi^2 \geq  C \|f\|_2^4,
\label{eq18}
\end{equation}
where $S^{(1)}, ~S^{(2)} \in Sp(2n, \mathbb{R})$ and $C>0$ is a positive constant which depends only on $S^{(1)}, ~S^{(2)} $ and $n$.

A straightforward way to obtain the uncertainty principle in the LCT domain is suggested in \cite{Jaming}. Let us briefly discuss this approach. We consider the one-dimensional case. For simplicity, we shall assume that the expectation values vanish, $<\xi>=<x>=0$. Let $g(x)= \gamma_2(x)f(x)$, where $\gamma_2$ is defined in (\ref{eqExtraLCT2}). From (\ref{eqExtraLCT1},\ref{eqExtraLCT2},\ref{eqMetaplectic4}), we have:
\begin{equation}
\begin{array}{c}
\Delta_S \xi^2 = \int_{\mathbb{R}}\xi^2 |\widehat{f}_S(\xi)|^2 d \xi=\frac{1}{|b|}\int_{\mathbb{R}}\xi^2 \left|\widehat{g}\left(\frac{\xi}{b}\right) \right|^2 d \xi=\\
\\
=b^2  \int_{\mathbb{R}}\xi^2 |\widehat{g}(\xi)|^2 d \xi=b^2 \|~|\xi| \widehat{g} \|_2^2~.
\end{array}
\label{eqHUPExtra1}
\end{equation}
On the other hand we also have:
\begin{equation}
\Delta x^2 = \| ~ |x| f\|_2^2=\| ~ |x| g\|_2^2~, \hspace{0.5 cm} \|f\|_2=\|g\|_2~.
\label{eqHUPExtra2}
\end{equation}
From the Heinsenberg uncertainty principle for the Fourier transform, we have:
\begin{equation}
\|~|x| g\|_2^2 ~ \|~|\xi| \widehat{g}\|_2^2 \geq \frac{\|g\|_2^4}{16 \pi^2}~.
\label{eqHUPExtra3}
\end{equation}
Substituting (\ref{eqHUPExtra1},\ref{eqHUPExtra2}) in (\ref{eqHUPExtra3}), we obtain:
\begin{equation}
\Delta x^2 \Delta_S \xi^2 \geq \frac{b^2 \|f\|_2^4}{16 \pi^2}~,
\label{eqHUPExtra4}
\end{equation}
which is the uncertainty principle obtained by Stern \cite{Stern}.

This is certainly an expedite way to obtain the uncertainty principle in the LCT domain. However, this approach does not work well in the multidimensional case. Indeed, if we follow the previous approach for $n \geq 2$, we will obtain the following inequality:
\begin{equation}
Tr \left[B^{-1} \text{Cov}_{\xi} (\widehat{f}_S) (B^{-1})^T \right] ~\| |x|~ f\|_2^2 \geq \frac{n^2 \|f\|_2^4}{16 \pi^2}~,
\label{eqHUPExtra5}
\end{equation}
where $\text{Cov}_{\xi} (\widehat{f}_S)$ is the matrix with entries:
\begin{equation}
\left[\text{Cov}_{\xi} (\widehat{f}_S)\right]_{jk}= \int_{\mathbb{R}}   \xi_j \xi_k |\widehat{f}_S(\xi)|^2 d \xi ~.
\label{eqHUPExtra6}
\end{equation}
Admittedly, inequality (\ref{eqHUPExtra5}) can be viewed as an uncertainty principle, but it is not the usual Heisenberg uncertainty principle in terms of variances of $x$ and $\xi$, which would be the analogue of (\ref{eq13}) or (\ref{eq15}). Moreover, it is asymmetric in $x$ and $\xi$, because it uses the variance as a measure of dispersion for $x$ but a totally different measure of dispersion for $\xi$.

Finally, let us also remark that with this direct approach, the optimal constant on the right-hand side of (\ref{eqHUPExtra4}) depends only on one of the entries of the symplectic matrix $S$ and it is not clear why this happens.

Here we wish to devise an alternative method which makes full use of the symplectic invariance associated with the metaplectic operators and permits us to obtain the Heisenberg uncertainty principle in the LCT domain in arbitrary dimension. Moreover, we obtain the dependence of the optimal constant on the full symplectic matrix.

\begin{theorem}\label{TheoremHUPLCT}
For any $j,k=1, \cdots,n$, constants $a_j, b_k \in \mathbb{R}$ and all $f \in L^2 (\mathbb{R}^d)$ the following inequality holds:
\begin{equation}
\begin{array}{c}
\int_{\mathbb{R}^n} (x_j -a_j)^2 | (\widehat{S^{(1)}}f)(x)|^2 d x ~\int_{\mathbb{R}^n} (\xi_k -b_k)^2 | (\widehat{S^{(2)}}f) (\xi)|^2 d \xi  \geq \\
\\
\geq \frac{\|f\|_2^4}{16 \pi^2} \left|\left(S^{(1)} J \left(S^{(2)} \right)^T \right)_{jk} \right|^2,
\end{array}
\label{eq19A}
\end{equation}
where $S^{(1)},S^{(2)} \in Sp(2n, \mathbb{R})$.

Moreover, this inequality is sharp.
\end{theorem}

\begin{proof}
If either $\int_{\mathbb{R}^n} (x_j -a_j)^2 | (\widehat{S^{(1)}}f) (x)|^2 d x= \infty$ or $ \int_{\mathbb{R}^n} (\xi_k -b_k)^2 | (\widehat{S^{(2)}}f) (\xi)|^2 d \xi  = \infty$, then there is nothing to prove. So we shall assume that both are finite.

To prove (\ref{eq19A}), we shall apply (\ref{eq11}) to the operators
\begin{equation}
\widehat{A}=\sum_{\alpha=1}^{2n} S^{(1)}_{j, \alpha} \widehat{Z}_{\alpha}, \hspace{1 cm}
\widehat{B}=\sum_{\beta=1}^{2n} S^{(2)}_{k, \beta} \widehat{Z}_{\beta},
\label{eq20}
\end{equation}
for some fixed $j,k=1, \cdots, n$ and where the operators $\widehat{Z}_{\alpha}$ are given by (\ref{eq21}).

We have from (\ref{eq21A}):
\begin{equation}
\begin{array}{c}
\left[\widehat{A},\widehat{B} \right]= \sum_{1 \leq \alpha, \beta \leq 2n} S^{(1)}_{j, \alpha}S^{(2)}_{k, \beta} \left[\widehat{Z}_{\alpha},\widehat{Z}_{\beta} \right] = \\
\\
=\frac{i}{2 \pi} \sum_{1 \leq \alpha, \beta \leq 2n} S^{(1)}_{j, \alpha}S^{(2)}_{k, \beta} J_{\alpha \beta}= \frac{i}{2 \pi} \left(S^{(1)} J \left(S^{(2)}\right)^T \right)_{jk}.
\end{array}
\label{eq22}
\end{equation}
On the other hand, we have from (\ref{eqCovariance}) and (\ref{eq20}) that:
\begin{equation}
\widehat{A}= \left(\widehat{S^{(1)}}\right)^{\ast} \widehat{X}_j \widehat{S^{(1)}}~, \hspace{0.5 cm} \widehat{B}= \left(\widehat{S^{(2)}}\right)^{\ast} \widehat{X}_k \widehat{S^{(2)}}~. 
\label{eq22A}
\end{equation}
Since $\widehat{S^{(1)}}$ and $\widehat{S^{(2)}}$ are unitary operators, we conclude that:
\begin{equation}
\begin{array}{c}
\|(\widehat{A}-a_j)f \|_2^2 = \| \left(\left(\widehat{S^{(1)}}\right)^{\ast} \widehat{X}_j \widehat{S^{(1)}} -a_j \right)f \|_2^2= \\
\\
=\|\left(\widehat{X}_j -a_j \right)\widehat{S^{(1)}}f \|_2^2= \int_{\mathbb{R}^n} |x_j-a_j|^2 |(\widehat{S^{(1)}}f) (x)|^2 d x~,
\end{array}
\label{eq23}
\end{equation}
and
\begin{equation}
\begin{array}{c}
\|(\widehat{B}-b_k)f \|_2^2 = 
\| \left(\left(\widehat{S^{(2)}}\right)^{\ast} \widehat{X}_k \widehat{S^{(2)}} -b_k \right)f \|_2^2= \\
\\
=\|\left(\widehat{X}_k -b_k \right)\widehat{S^{(2)}}f \|_2^2= \int_{\mathbb{R}^n} |\xi_k-b_k|^2 |(\widehat{S^{(2)}}f) (\xi)|^2 d \xi~.
\end{array}
\label{eq24}
\end{equation}
From (\ref{eq11},\ref{eq22},\ref{eq23},\ref{eq24}), we recover (\ref{eq19A}).

That this inequality is sharp is an immediate consequence of Proposition \ref{PropositionHeisenberg}. To obtain the optimizers, one just has to apply (\ref{eq12}). 
\end{proof}

\begin{remark}\label{Remark0}
Notice that the author of \cite{Zhang1,Zhang2} already used a metaplectic approach to the uncertainty principles in the LCT domain and obtained interesting results. However, we feel that our approach is more geometric and leads to a much more compact form of these uncertainty principles for all signals (real or complex), in arbitrary dimension and it holds for all metaplectic transformations and not just linear canonical transforms.
\end{remark}

\begin{remark}\label{Remark1}
As a particular case, if we take $S^{(1)}=I_n$ and $S^{(2)}=J$, we recover:
\begin{equation}
\int_{\mathbb{R}^n}(x_j-a_j)^2 |f(x)|^2 dx ~ \int_{\mathbb{R}^n}(\xi_k-b_k)^2 |\widehat{f}( \xi)|^2 d \xi  \geq \frac{\|f\|_2^4}{16 \pi^2} \delta_{j,k},
\label{eq30}
\end{equation}
as expected.

Another example is in dimension $n=1$, with $S^{(1)}=I_2$ and
\begin{equation}
S^{(2)}= \left(
\begin{array}{c c}
a& b\\
c& d
\end{array}
\right)~,
\label{eqExample1}
\end{equation}
with $b \neq 0$ and $ad-bc =1$.

We then have from eq.(\ref{eq19}) that:
\begin{equation}
\Delta x^2 ~ \Delta_{S^{(2)}} \xi^2 \geq \frac{b^2}{16 \pi^2} \| f\|_2^4,
\label{eqExample1}
\end{equation}
which is the result obtained by Stern \cite{Stern}.
\end{remark}

As a Corollary we obtain the following $n$-dimensional version of the uncertainty principle in the LCT domain.

\begin{corollary}\label{Corollary1}
Let $f \in L^2 (\mathbb{R}^n)$. Then we have:
\begin{equation}
\Delta_{S^{(1)}} x \Delta_{S^{(2)}} \xi \geq  \frac{\|f\|_2^2}{4 \pi} \sum_{j=1}^n \left|\left( S^{(1)} J \left(S^{(2)} \right)^T \right)_{jj} \right|.
\label{eq31}
\end{equation}
\end{corollary}

\begin{proof}
If we set $j=k$ in (\ref{eq19A}), take the square root and add over $j$, we obtain:
\begin{equation}
\begin{array}{c}
\sum_{j=1}^n \left( \int_{\mathbb{R}^n} (x_j -a_j)^2 | (\widehat{S^{(1)}}f) (x)|^2 d x \right)^{1/2} \left(\int_{\mathbb{R}^n} (\xi_j -b_j)^2 | (\widehat{S^{(2)}}f) (\xi)|^2 d \xi  \right)^{1/2} \geq \\
 \\
 \geq \frac{\|f\|_2^2}{4 \pi} \sum_{j=1}^n  \left|\left(S^{(1)} J \left(S^{(2)} \right)^T \right)_{jj} \right|
\end{array}
\label{eq32}
\end{equation}
By the Cauchy-Schwarz inequality, the left-hand side of (\ref{eq32}) is dominated by
\begin{equation}
\begin{array}{c}
\Delta_{S^{(1)}} x \Delta_{S^{(2)}} \xi=  \\
\\
=\left( \sum_{j=1}^n  \int_{\mathbb{R}^n} (x_j -a_j)^2 | (\widehat{S^{(1)}}f) (x)|^2 d x \right)^{1/2} \left(\sum_{k=1}^n \int_{\mathbb{R}^n} (\xi_k -b_k)^2 | (\widehat{S^{(2)}}f) (\xi)|^2 d \xi  \right)^{1/2}
\end{array}
\label{eq33}
\end{equation}
and we recover (\ref{eq31}).
\end{proof}

\section{Uncertainty principle with covariance matrices}

The Heisenberg uncertainty principle does not include the time-frequency correlations. Thus its form will change in general under a linear canonical transformation. The Robertson-Schr\"odinger uncertainty principle \cite{Gosson} on the other hand is invariant under linear symplectic transformations and is stronger than the Heisenberg uncertainty principle.

Let $\Sigma=(\Sigma_{\alpha,\beta})$ be the covariance matrix
\begin{equation}
\begin{array}{c}
\Sigma_{\alpha, \beta}=\langle \left(\frac{\widehat{Z}_{\alpha}\widehat{Z}_{\beta}+\widehat{Z}_{\beta} \widehat{Z}_{\alpha}}{2}\right) f,f \rangle=\\
\\
= \int_{\mathbb{R}^{2n}} z_{\alpha}z_{\beta}W_{\sigma} f (z) dz, \hspace{0.5 cm} \alpha, \beta=1, \cdots, 2n~,
\end{array}
\label{eq38}
\end{equation}
where the operators $\widehat{Z}_{\alpha}$ are as in (\ref{eq21},\ref{eq21A}), and where we assumed that the expectation value of $\widehat{Z}_{\alpha}$ is $<\widehat{Z}_{\alpha}>=0$. This can be easily achieved by a phase-space translation. We have also assumed that the function $f$ is normalized: $\|f\|_2=1$.

The Robertson-Schr\"odinger uncertainty principle is expressed in terms of the following matrix inequality:
\begin{equation}
\Sigma + \frac{i}{4 \pi} J \geq 0~.
\label{eq41}
\end{equation}
It is easy to check that this inequality is invariant under the linear symplectic transformation:
\begin{equation}
z \mapsto Sz~,\hspace{1 cm} S \in  \operatorname*{Sp}(2n,\mathbb{R})~.
\label{eq38A}
\end{equation}
Moreover it is stronger than the Heisenberg uncertainty principle (\ref{eq15}). To illustrate this let us consider the $n=1$ case. We set:
\begin{equation}
\Sigma_{1,1}= \Delta x^2~, \hspace{0.5 cm} \Sigma_{2,2}= \Delta \xi^2~. 
\label{eq38B}
\end{equation}
From (\ref{eq41}), we have:
\begin{equation}
\det \left(
\begin{array}{c c}
\Delta x^2 & \Sigma_{1,2}+\frac{i}{4\pi}\\
\Sigma_{1,2}-\frac{i}{4 \pi} &\Delta \xi^2  
\end{array}
\right) \geq 0      \Leftrightarrow \Delta x^2\Delta \xi^2 \geq  \Sigma_{1,2}^2+ \frac{1}{16 \pi^2}~,
\label{eq38C}
\end{equation}
where $\Sigma_{1,2}= \langle \left(\frac{\widehat{X}\widehat{P}+\widehat{P} \widehat{X}}{2}\right) f,f \rangle$, and where we used the fact that $\Sigma_{1,2}=\Sigma_{2,1}$.
This inequality is manifestly tighter than (\ref{eq15}).

Let us now see how this can be generalized for our framework.

The uncertainty principle (\ref{eq31}) does not include the correlations between the variables $\xi_j^{(1)}= \sum_{\alpha=1}^{2n} S^{(1)}_{j, \alpha} z_{\alpha}$ and $\xi_k^{(2)}= \sum_{\beta=1}^{2n} S^{(2)}_{k, \beta} z_{\beta}$. Thus the form of (\ref{eq31}) will change in general under a linear coordinate transformation. In analogy with the Robertson-Schr\"odinger uncertainty principle, we will now derive an uncertainty principle involving the correlations of $\xi_j^{(1)}$ and $\xi_j^{(2)}$ and which is invariant under a specific group of affine linear coordinate transformations. Moreover, it is stronger than (\ref{eq31}). As previously, we shall always assume that the signal is normalized $\|f\|_2=1$.

Let us define the joint variable $\gamma=(\xi^{(1)},\xi^{(2)})= \left( \xi_1^{(1)}, \cdots,\xi_n^{(1)}, \xi_1^{(2)}, \cdots,\xi_n^{(2)}\right) \in \mathbb{R}^{2n}$, which is the Weyl-symbol of the operator $\widehat{\Gamma}= (\widehat{\Gamma}_1,\cdots, \widehat{\Gamma}_{2n})$:
\begin{equation}
\widehat{\Gamma}_j= \sum_{\alpha=1}^{2n} S^{(1)}_{j, \alpha}
\widehat{Z}_{\alpha} ~,  \hspace{0.5 cm} \widehat{\Gamma}_{n+k}= \sum_{\beta=1}^{2n} S^{(2)}_{k, \beta} \widehat{Z}_{\beta}~, ~j,k=1, \cdots, n~.
\label{eq33A}
\end{equation} 

If we set (as in eq.(\ref{eqSymp3.1})),
\begin{equation}
S^{(1)}= \left(
\begin{array}{c c}
A^{(1)} &B^{(1)}\\
C^{(1)} & D^{(1)}
\end{array}
\right)~, \hspace{0.5 cm} S^{(2)}= \left(
\begin{array}{c c}
A^{(2)} &B^{(2)}\\
C^{(2)} & D^{(2)}
\end{array}
\right)~, 
\label{eqMatrices}
\end{equation}
then we can write:
\begin{equation}
\gamma = D^{(1,2)} z,
\label{eq34}
\end{equation}
where
\begin{equation}
 D^{(1,2)} =\left(
 \begin{array}{c c}
 A^{(1)} & B^{(1)}\\
 & \\
 A^{(2)} &  B^{(2)}
 \end{array}
 \right).
\label{eq35}
\end{equation}
Since $<\gamma> = D^{(1,2)} <z>$, we may, as before, perform a time-frequency translation such that $<\gamma>= <z>=0$. The covariance matrix $\Upsilon= (\Upsilon_{\alpha \beta})$ of the variables $\gamma$ may thus be computed as:
\begin{equation}
\begin{array}{c}
\Upsilon_{\alpha \beta}= \langle \left(\frac{\widehat{\Gamma}_{\alpha}\widehat{\Gamma}_{\beta}+ \widehat{\Gamma}_{\beta}\widehat{\Gamma}_{\alpha}}{2}\right) f,f \rangle =\\
\\
= \int_{\mathbb{R}^{2n}} \gamma_{\alpha}\gamma_{\beta}W_{\sigma} f (z) dz = \sum_{1 \leq \varpi, \lambda \leq 2n} D_{\alpha, \tau}^{(1,2)}D_{\beta, \lambda}^{(1,2)} \int_{\mathbb{R}^{2n}} z_{\tau}z_{\lambda}W_{\sigma}f (z) dz.
\end{array}
\label{eq36}
\end{equation}
That is:
\begin{equation}
\Upsilon = D^{(1,2)} \Sigma \left(D^{(1,2)}\right)^T,
\label{eq37}
\end{equation}

We then have the following uncertainty principle.
\begin{theorem}\label{TheoremRobertsonSchrodinger}
Let $f$ be such that $\|f\|_2=1$ and $\int_{\mathbb{R}^{2n}} (1+ |z|^2) W_{\sigma} f(z) < \infty$. Let $\Upsilon$ be the matrix defined as in (\ref{eq37}). Then we have the following matrix inequality:
\begin{equation}
\Upsilon + \frac{i}{4 \pi} \xi \geq 0,
\label{eq39}
\end{equation}
where
\begin{equation}
\Omega= D^{(1,2)} J \left(D^{(1,2)} \right)^T.
\label{eq40}
\end{equation}
\end{theorem}

\begin{proof}
The covariance matrix $\Sigma$ satisfies the Robertson-Schr\"odinger uncertainty principle (\ref{eq41}).
From (\ref{eq41}), (\ref{eq37}) and (\ref{eq40}), we conclude that $\Upsilon$ satisfies (\ref{eq39}).
\end{proof}

\begin{remark}\label{Remark2}
As a particular example, let us consider again the one dimensional case with
\begin{equation}
S^{(p)} = \left(
\begin{array}{c c}
a_p & b_p\\
c_p & d_p
\end{array}
\right)~,
\label{eq42}
\end{equation}
for $p=1,2$, and
\begin{equation}
D^{(1,2)} = \left(
\begin{array}{c c}
a_1 & b_1\\
a_2 & b_2
\end{array}
\right)~.
\label{eq42A}
\end{equation}
A simple calculation yields:
\begin{equation}
\Omega = (a_2 b_1-a_1 b_2 ) J.
\label{eq43}
\end{equation}
Consequently, (\ref{eq39}) reads:
\begin{equation}
\left(
\begin{array}{c c}
\Upsilon_{1,1} & \Upsilon_{1,2}+ \frac{i}{4\pi} (a_2 b_1-a_1 b_2 ) \\
& \\
\Upsilon_{1,2}- \frac{i}{4\pi} (a_2 b_1-a_1 b_2 ) & \Upsilon_{2,2}
\end{array}
\right) \geq 0.
\label{eq44}
\end{equation}
Since $\Upsilon_{1,1}= \left(\Delta \xi^{(1)}\right)^2$ and $\Upsilon_{2,2}= \left(\Delta \xi^{(2)}\right)^2$, we obtain from (\ref{eq44}) that:
\begin{equation}
 \left(\Delta \xi^{(1)}\right)^2 ~ \left(\Delta \xi^{(2)}\right)^2 \geq \left(\Upsilon_{1,2}\right)^2 + \frac{1}{16 \pi^2} (a_2 b_1-a_1 b_2 )^2,
\label{eq45}
\end{equation}
which is the result of \cite{Guanlei3}.
\end{remark}

\section{An interpretation in terms of non-standard symplectic vector spaces}

It is interesting and possibly fruitful to interpret inequalities such as (\ref{eq19A}) or (\ref{eq39}) as arising from a set of generalized coordinates which do not satisfy the ordinary canonical commutation relations. Let us consider the canonical variables $\widehat{Z}= \left( \widehat{X},\widehat{P} \right)$ as defined in (\ref{eq21},\ref{eq21A}).

On the other hand, let us consider the variables $(\Gamma_{\alpha})$ as in (\ref{eq34}):
\begin{equation}
\widehat{\Gamma}_{\alpha} = \sum_{\beta =1}^{2n} D_{\alpha,\beta}^{(1,2)} \widehat{Z}_{\beta}, \hspace{1 cm} \alpha=1, \cdots, 2n.
\label{eq48}
\end{equation}
In that case we obtain the commutation relations:
\begin{equation}
\left[\widehat{\Gamma}_{\alpha},\widehat{\Gamma}_{\beta} \right]= \frac{i}{2 \pi} \Omega_{\alpha, \beta}, \hspace{1 cm} \alpha, \beta=1, \cdots,2n,
\label{eq49}
\end{equation}
where the matrix $\Omega= \left( \Omega_{\alpha, \beta} \right)$ is given by (\ref{eq40}).

If $\det D^{(1,2)} \neq 0$, then the matrix $\Omega$ defines a symplectic form. We have the two symplectic forms on $\mathbb{R}^{2n}$:
\begin{equation}
\sigma (z, z^{\prime})= z \cdot J^{-1} z, \hspace{0.5 cm} \text{ and } \hspace{0.5 cm}\vartheta (z, z^{\prime})= z \cdot \Omega^{-1} z,
\label{eq50}
\end{equation}
for $z,z^{\prime} \in \mathbb{R}^{2n}$, and $D^{(1,2)}$ is a corresponding Darboux matrix. Recall that (see section \ref{SectionSymGeom}), the $\vartheta$-symplectic group $\operatorname*{Sp}(2n,\vartheta)$ is the group of $2n \times 2n$ real matrices $P$ that satisfy (\ref{eqSymp7}). 

We then have:

\begin{proposition}\label{Proposition1}
The uncertainty principle (\ref{eq39}) is invariant under the linear transformation
\begin{equation}
\gamma \mapsto P \gamma,
\label{eq54}
\end{equation}
where $P \in \operatorname*{Sp}(2n,\vartheta)$.
\end{proposition}

\begin{proof}
Under the linear transformation (\ref{eq54}), we have:
\begin{equation}
\Upsilon \mapsto P \Upsilon P^T.
\label{eq55}
\end{equation}
Consequently:
\begin{equation}
P \Upsilon P^T + \frac{i}{4 \pi} \Omega = P \left(\Upsilon + \frac{i}{4 \pi} P^{-1}\Omega (P^{-1})^T \right) P^T= P \left(\Upsilon + \frac{i}{4 \pi} \Omega \right) P^T \geq 0,
\label{eq56}
\end{equation}
which proves the result.
\end{proof}

From section \ref{SectionWeylNonStandard}, we have:

\begin{definition}\label{DefinitionNewBilinearDist}
The Wigner function on the symplectic vector space $\left(\mathbb{R}^{2n}; \vartheta \right)$ associated with a given signal $f \in L^2 (\mathbb{R}^n)$ is given by:
\begin{equation}
W_{\vartheta} f (z)= \frac{1}{\sqrt{\det \Omega}} W_{\sigma} f\left( \left(D^{(1,2)} \right)^{-1} z \right)~.
\label{eq56}
\end{equation}
\end{definition}

This definition provides a practical interpretation for the problems we addressed in this paper. We just have to think that we now have a Wigner function $W_{\vartheta} f$ defined for a different symplectic structure. Thus, for example, we may compute the covariance matrix $\Upsilon$ as we have done in (\ref{eq36}):
\begin{equation}
\Upsilon_{\alpha \beta}= \int_{\mathbb{R}^{2n}} \gamma_{\alpha}\gamma_{\beta}W_{\sigma} f (z) dz = \sum_{1 \leq \tau, \lambda \leq 2n} D_{\alpha, \tau}^{(1,2)}D_{\beta, \lambda}^{(1,2)} \int_{\mathbb{R}^{2n}} z_{\tau}z_{\lambda}W_{\sigma} f (z) dz,
\label{eq57}
\end{equation}
or, alternatively, we may evaluate it as an ordinary covariance matrix:
\begin{lemma}\label{LemmaCovarianceMatrix}
The covariance matrix (\ref{eq57}) can be evaluated as an ordinary covariance matrix for the bilinear distribution $W_{\vartheta}f$:
\begin{equation}
\Upsilon_{\alpha \beta}= \int_{\mathbb{R}^{2n}} z_{\alpha}z_{\beta}W_{\vartheta} f (z) dz .
\label{eq58}
\end{equation}
\end{lemma}

\begin{proof}
The equivalence of the two expression in (\ref{eq57}) and (\ref{eq58}) is easily established, using the fact that $\sqrt{\det \Omega}= \left|\det D^{(1,2)} \right|$:
\begin{equation}
\begin{array}{c}
 \int_{\mathbb{R}^{2n}} z_{\alpha}z_{\beta}W_{\vartheta} f (z) dz = \\
 \\
 =\frac{1}{\sqrt{\det \Omega}}  \int_{\mathbb{R}^{2n}} z_{\alpha}z_{\beta}W_{\sigma} f \left(\left(D^{(1,2)} \right)^{-1} z \right)) dz = \\
 \\
 =\int_{\mathbb{R}^{2n}} \left(D^{(1,2)}z \right)_{\alpha} \left(D^{(1,2)}z \right)_{\beta}W_{\sigma} f (z) dz=\\
 \\
  =\int_{\mathbb{R}^{2n}} \gamma_{\alpha}\gamma_{\beta}W_{\sigma} f (z) dz.
\end{array}
\label{eq59}
\end{equation}
\end{proof}

All the other relevant quantities associated  with $\widehat{S^{(1)}}f$ and $\widehat{S^{(2)}}f$ can also be obtained from $W_{\vartheta} f$, as illustrated by the following theorem.
  
\begin{theorem}\label{TheoremMarginal}
Let $f \in \mathcal{S} (\mathbb{R}^n)$.
We have:
\begin{enumerate}
\item \begin{equation}
\int_{\mathbb{R}^n}\int_{\mathbb{R}^n} x^{\alpha} W_{\vartheta} f (x,\xi) d x d\xi = \int_{\mathbb{R}^n}  x^{\alpha} |(\widehat{S^{(1)}}f) (x)|^2 d x~, 
\label{eq59A}
\end{equation}
for all multi-indices $\alpha=(\alpha_1, \cdots, \alpha_n) \in \mathbb{N}_0^n$, with $|\alpha|= \alpha_1 + \cdots +\alpha_n >0$.

\item \begin{equation}
\int_{\mathbb{R}^n}\int_{\mathbb{R}^n} \xi^{\alpha} W_{\vartheta} f (x,\xi) dx d \xi = \int_{\mathbb{R}^n}  \xi^{\alpha} |(\widehat{S^{(2)}}f) (\xi)|^2 d \xi~, 
\label{eq59B}
\end{equation}
for all multi-indices $\alpha=(\alpha_1, \cdots, \alpha_n) \in \mathbb{N}_0^n$, with $|\alpha|= \alpha_1 + \cdots\alpha_n >0$.

\item \begin{equation}
\int_{\mathbb{R}^n}  W_{\vartheta} f (x,\xi) d \xi  =  |(\widehat{S^{(1)}}f) (x)|^2 ~, 
\label{eq59C}
\end{equation}
for all $ x\in \mathbb{R}^n$.

\item \begin{equation}
\int_{\mathbb{R}^n} W_{\vartheta} f (x,\xi) d x  =  |(\widehat{S^{(2)}}f)(\xi)|^2 ~, 
\label{eq59D}
\end{equation}
for all $ \xi \in \mathbb{R}^n$.

\end{enumerate}
\end{theorem}

\begin{proof}
Let $a(z)=a(x,\xi) \in \mathcal{S}^{\prime} (\mathbb{R}^{2n})$ be some symbol and let $S \in \operatorname*{Sp}(2n,\mathbb{R})$. We have:
\begin{equation}
\begin{array}{c}
\int_{\mathbb{R}^n} \int_{\mathbb{R}^n} a(x,\xi) W_{\vartheta} f(x,\xi) dx d \xi=\int_{\mathbb{R}^{2n}}  a(z) W_{\vartheta} f(z) d z =\\
\\
= \frac{1}{\sqrt{\det \Omega}} \int_{\mathbb{R}^{2n}}  a(\zeta) W_{\sigma} f\left((D^{(1,2)})^{-1} \zeta \right) d \zeta=\\
\\
=\int_{\mathbb{R}^{2n}}  a\left(D^{(1,2)} z_0 \right) W_{\sigma} f(z_0) d z_0 = \int_{\mathbb{R}^{2n}}  a\left(D^{(1,2)} S^{-1} z \right) W_{\sigma} f\left(S^{-1} z \right) d z=\\
\\
= \int_{\mathbb{R}^{2n}}  a\left(D^{(1,2)} S^{-1} z \right) W_{\sigma}( \widehat{S}f)(z) d z~.
\end{array}
\label{eqMarginal1}
\end{equation}
In the sequel we consider the following  projections $\Pi_j : \mathbb{R}^{2n} \to \mathbb{R}^n$, $j=1,2$:
\begin{equation}
\Pi_1 (z) = (z)_1= x ~,  \hspace{1 cm}  \Pi_2 (z) = (z)_2= \xi ~,
\label{eqMarginal1A}
\end{equation}
for any $z=(x, \xi) \in \mathbb{R}^{2n}$. 

\begin{enumerate}
\item In the identity (\ref{eqMarginal1}) set $a(z)=a(x,\xi) =(z)_1^{\alpha} =x^{\alpha}$ and $S=S^{(1)}$. We obtain:
\begin{equation}
\int_{\mathbb{R}^n} \int_{\mathbb{R}^n} x^{\alpha} W_{\vartheta} f(x,\xi) dx d \xi=\int_{\mathbb{R}^{2n}}  \left(D^{(1,2)} (S^{(1)})^{-1} z \right)_1^{\alpha} W_{\sigma} (\widehat{S^{(1)}}f) (z) d z~.
\label{eqMarginal2}
\end{equation}
Notice that (cf.(\ref{eq35})):
\begin{equation}
\left(D^{(1,2)} (S^{(1)})^{-1} z \right)_1^{\alpha}=\left(S^{(1)} (S^{(1)})^{-1} z \right)_1^{\alpha}= \left(z \right)_1^{\alpha}= x^{\alpha}~,
\label{eqMarginal3}
\end{equation}
for $z= (x,\xi) \in \mathbb{R}^{2n}$.

Thus:
\begin{equation}
\begin{array}{c}
\int_{\mathbb{R}^n} \int_{\mathbb{R}^n} x^{\alpha} W_{\vartheta} f(x,\xi) dx d \xi =\int_{\mathbb{R}^{2n}}  \left(z \right)_1^{\alpha} W_{\sigma}(\widehat{S^{(1)}}f) (z) d z=\\
\\
=\int_{\mathbb{R}^{n}} \int_{\mathbb{R}^{n}}  x^{\alpha} W_{\sigma} (\widehat{S^{(1)}}f)(x,\xi) dx d \xi= \int_{\mathbb{R}^{n}}  x^{\alpha} |(\widehat{S^{(1)}}f) (x)|^2 d x~.
\end{array}
\label{eqMarginal4}
\end{equation}

\item The result is obtained similarly for $a(\zeta)=a(x,\xi) =(\zeta)_2^{\alpha} =\xi^{\alpha}$ and $S=S^{(2)}$. 

\item In the identity (\ref{eqMarginal1}) set $a(z^{\prime})=a(x^{\prime}, \xi^{\prime}) =\delta \left( (z- z^{\prime})_1\right) =\delta(x-x^{\prime})$, where $z=(x,\xi)$, and let $S=S^{(1)}$. We obtain:
\begin{equation}
\begin{array}{c}
\int_{\mathbb{R}^{n}}W_{\vartheta} f (x,\xi) d \xi= \int_{\mathbb{R}^{n}}\int_{\mathbb{R}^{n}}\delta(x-x^{\prime}) W_{\vartheta} f (x^{\prime}, \xi^{\prime}) dx^{\prime} d \xi^{\prime}=\\
\\
= \int_{\mathbb{R}^{2n}} \delta \left( (z- z^{\prime})_1\right) W_{\vartheta} f (z^{\prime}) d z^{\prime} = \\
\\
=\int_{\mathbb{R}^{2n}} \delta \left( (z- D^{(1,2)} (S^{(1)})^{-1} z^{\prime})_1\right) W_{\sigma}(\widehat{S^{(1)}}f) (z^{\prime}) d z^{\prime}~.
\end{array}
\label{eqMarginal5}
\end{equation}
Again, we have:
\begin{equation}
\begin{array}{c}
(z- D^{(1,2)} (S^{(1)})^{-1} z^{\prime})_1 =(z)_1- (D^{(1,2)} (S^{(1)})^{-1} z^{\prime})_1=\\
\\
=(z)_1- (S^{(1)} (S^{(1)})^{-1} z^{\prime})_1=(z)_1- (z^{\prime})_1=(z- z^{\prime})_1
\end{array}
\label{eqMarginal6}
\end{equation}
Thus:
\begin{equation}
\begin{array}{c}
\int_{\mathbb{R}^{n}}W_{\vartheta} f (x,\xi) d \xi=\int_{\mathbb{R}^{2n}} \delta \left( (z-z^{\prime})_1\right) W_{\sigma} (\widehat{S^{(1)}}f) (z^{\prime}) d z^{\prime}=\\
\\
=\int_{\mathbb{R}^{n}}\int_{\mathbb{R}^{n}} \delta (x-x^{\prime}) W_{\sigma}(\widehat{S^{(1)}}f) (x^{\prime}, \xi^{\prime}) d x^{\prime} d \xi^{\prime}=\\
\\
= \int_{\mathbb{R}^{n}} W_{\sigma} (\widehat{S^{(1)}}f) (x,\xi)  d \xi= |(\widehat{S^{(1)}}f)(x)|^2 ~,
\end{array}
\label{eqMarginal7}
\end{equation}
for all $x\in \mathbb{R}^n$.

\item The result is obtained following the same procedure with $a(z^{\prime})=a(x^{\prime}, \xi^{\prime}) =\delta \left( (z- z^{\prime})_2\right) =\delta(\xi-\xi^{\prime})$, where $z=(x,\xi)$, and $S=S^{(2)}$.
\end{enumerate}
\end{proof}

\begin{remark}\label{RemarkMoyalMarginal}
The previous results are not at all obvious, and they are definitely not immediate consequences of the definition (\ref{eq56}) above of the bilinear distribution. Indeed, in some of our previous work (see e.g. refs. \cite{Dias1}, \cite{Dias2}) for specific symplectic forms $\vartheta$ (associated with so-called noncommutative quantum mechanics), we obtain marginal distributions of the form (in dimension $n=2$):
\begin{equation}
\int_{\mathbb{R}^2} W_{\vartheta} f(x, \xi) d\xi = f(x) \star_M \overline{f (x)}~,
\label{eq18}
\end{equation}
where $\star_M$ is a Moyal product:
\begin{equation}
\begin{array}{c}
f(x_1,x_2) \star_M g(x_1,x_2)= f(x_1,x_2) \cdot  g(x_1,x_2)+\\
\\
+ \frac{i \theta}{2}\left( \frac{\partial f}{\partial x_1} \frac{\partial g}{\partial x_2} - \frac{\partial f}{\partial x_2} \frac{\partial g}{\partial x_1} \right) +  \mathcal{O}(\theta^2)~,
\end{array}
\label{eq19}
\end{equation}
where $\theta >0$ is some physical parameter. Thus, the marginal distribution (\ref{eq18}) can become negative in general.  
\end{remark}

\begin{remark}\label{RemarkRadon}
The marginal distributions appearing the previous theorem can be interpreted in the following fashion. We have from (\ref{eqMarginal7}):
\begin{equation}
\begin{array}{c}
\int_{\mathbb{R}^{n}}W_{\vartheta} f (x,\xi) d \xi=\int_{\mathbb{R}^{2n}} \delta \left( (z-z^{\prime})_1\right) W_{\sigma}(\widehat{S^{(1)}}f) (z^{\prime}) d z^{\prime}
=\\
\\
=\int_{\mathbb{R}^{2n}} \delta \left( (z-z^{\prime})_1\right) W_{\sigma} f \left((S^{(1)})^{-1} z^{\prime}\right) d z^{\prime} =\\
\\
=\int_{\mathbb{R}^{2n}} \delta \left( (z-S^{(1)}z^{\prime})_1\right) W_{\sigma} f (z^{\prime}) d z^{\prime}=\\
\\
=\int_{\mathbb{R}^{n}}\int_{\mathbb{R}^{n}} \delta \left(x - A^{(1)} x^{\prime}-B^{(1)}\xi^{\prime} \right) W_{\sigma} f (x^{\prime}, \xi^{\prime}) d x^{\prime} d \xi^{\prime}~. 
\end{array}
\label{eqRadon1}
\end{equation}
The expression on the right-hand side amounts to the symplectic Radon transform of the Wigner function $W_{\sigma}f$ on the Lagrangian plane $\xi = A^{(1)} x-B^{(1)}\xi$ \cite{Gosson1}, and also to the probability density for the observable $A^{(1)} \widehat{X} -B^{(1)} \widehat{P}$. 

Radon transforms of Wigner functions are called quantum tomograms and play a fundamental r\^ole in the tomographic picture of quantum mechanics \cite{Ibort}.
 
\end{remark}

\begin{remark}\label{RemarkHudson}
Another interesting result that follows immediately from this analysis regards the "saturation" of the uncertainty principle (\ref{eq39}). It is well known that the Robertson-Schr\"odinger uncertainty principle (\ref{eq41}) is saturated, when we have $n$ directions of minimal uncertainty \cite{Birk,Gosson}. In other words all the symplectic invariants of the Covariance matrix $\Sigma$ (see (\ref{eq38})) are equal to $\frac{1}{4 \pi}$. The only solution are the Gaussian Wigner functions of the form:
\begin{equation}
W_{\sigma} f (z)= G_{M}(z)= 2^n e^{- 2 \pi (z-z_0) \cdot M (z-z_0)},
\label{eq60}
\end{equation}
where $z_0 \in \mathbb{R}^{2n}$ and $M$ is a real, symmetric, positive-definite, symplectic $2n \times 2n $ matrix.

Likewise the new uncertainty principle (\ref{eq39}) is saturated if and only if
\begin{equation}
\begin{array}{c}
W_{\vartheta} f (z)= \frac{1}{\sqrt{\det \Omega}} G_{M} \left( \left(D^{(1,2)} \right)^{-1} z\right) = \\
\\
=\frac{2^n}{\sqrt{\det \Omega}} e^{- 2 \pi (z-z_0) \cdot \left(\left(D^{(1,2)} \right)^{-1}\right)^T M \left(D^{(1,2)} \right)^{-1} (z-z_0)}.
\end{array}
\label{eq60}
\end{equation}

The minimal uncertainties and the directions of minimal uncertainty can be obtained via the so-called $\vartheta$-\textit{Williamson invariants}. For more details see \cite{Dias5}.  

One could also state a counterpart of Hudson's theorem for the Weyl quantization on $(\mathbb{R}^{2n}; \vartheta)$ \cite{Hudson}. A function $W_{\vartheta}f (z)$ is non-negative for all $z \in \mathbb{R}^{2n}$, if and only if it is of the form (\ref{eq60}), in which case $f$ is a generalized Gaussian. Thus, we obtain the same class of functions as in the standard case.
\end{remark}

We conclude by showing that the bilinear distribution (\ref{eq56}) cannot belong to Cohen's class \cite{Cohen}, but it can be, in certain cases, a linear perturbation of the Wigner distribution as defined in \cite{Cordero}.

A bilinear distribution $Q(f)$ is said to belong to \textit{Cohen's class} if there exists a kernel $\Phi \in \mathcal{S}^{\prime} (\mathbb{R}^{2n})$, such that:
\begin{equation}
Q(f) (z)=Q_{\Phi}(f) (z)= \left(\Phi \star W_{\sigma}f \right)(z) = \int_{\mathbb{R}^{2n}} \Phi (z-z^{\prime}) W_{\sigma} f (z^{\prime}) d z^{\prime}
\label{eqCohenclass1}
\end{equation}
for all $f \in \mathcal{S}(\mathbb{R}^n)$.

\begin{theorem}\label{TheoremCohenClass}
The bilinear distribution $W_{\vartheta} f$ (\ref{eq56}) belongs to Cohen's class if and only if $\vartheta= \sigma$. 
\end{theorem}

\begin{proof}
From the phase-space covariance property (\ref{eqReview25A}) of Wigner functions it follows immediately that:
\begin{equation}
Q_{\Phi}\left(\widehat{T}^{\sigma} (z_0) f\right) (z)=Q_{\Phi}(f) (z-z_0)~,
\label{eqCohenclass2}
\end{equation}
for any $z_0 \in \mathbb{R}^{2n}$. 

Since the Wigner distribution $W_{\vartheta} f$ satisfies (\ref{eqReview25B}) instead of (\ref{eqReview25A}), we conclude that it cannot be in Cohen's class. Indeed, we have:
\begin{equation}
\begin{array}{c}
W_{\vartheta} \left(\widehat{T}^{\sigma} (z_0) f\right) (z)=\frac{1}{\sqrt{\det \Omega}} W_{\sigma} \left(\widehat{T}^{\sigma} (z_0) f\right) \left(D^{-1} z\right) =\\
\\
=\frac{1}{\sqrt{\det \Omega}} W_{\sigma}f \left(D^{-1} z-z_0\right)=\frac{1}{\sqrt{\det \Omega}} W_{\sigma}f \left(D^{-1} \left(z-Dz_0\right)\right)=\\
\\
=W_{\vartheta}f \left(z-Dz_0\right)~.
\end{array}
\label{eqCohenclass3}
\end{equation}
For $D \neq I_{2n}$ this is different from (\ref{eqCohenclass2}) whenever $z_0 \neq 0$.

Thus, there cannot exist any $\Phi \in \mathcal{S}^{\prime} (\mathbb{R}^{2n})$, such that:
\begin{equation}
W_{\vartheta}f (z)= \left(\Phi \star W_{\sigma} f \right) (z)~. 
\label{eqCohenclass4}
\end{equation}
\end{proof}
\begin{remark}\label{RemarkNewCohenClass}
The previous result suggest that we may define a new generalized Cohen class associated with an arbitrary symplectic form $\vartheta$. Given some $\Phi \in \mathcal{S}^{\prime}  (\mathbb{R}^{2n})$, we define:
\begin{equation}
Q_{\Phi,\vartheta}(f)(z)= \left(\Phi \star W_{\vartheta} f\right) (z) = \int_{\mathbb{R}^{2n}} \Phi (z-z^{\prime}) W_{\vartheta} f (z^{\prime}) d z^{\prime} ~,
\label{eqNewCohenclass1}
\end{equation}
for all $f \in \mathcal{S}(\mathbb{R}^n)$. For $\vartheta= \sigma$, we recover the usual Cohen class. It is worth mentioning that, under these circumstances (cf.(\ref{eqReview25B})):
\begin{equation}
\begin{array}{c}
Q_{\Phi,\vartheta}(\widehat{T}^{\vartheta} (z_0)f)(z)=\int_{\mathbb{R}^{2n}} \Phi (z-z^{\prime}) W_{\vartheta} (\widehat{T}^{\vartheta} (z_0)f) (z^{\prime}) d z^{\prime}=\\
\\
=\int_{\mathbb{R}^{2n}} \Phi (z-z^{\prime}) W_{\vartheta} f (z^{\prime}-z_0) d z^{\prime}=\\
\\
=\int_{\mathbb{R}^{2n}} \Phi (z-z_0-z^{\prime \prime}) W_{\vartheta} f (z^{\prime \prime}) d z^{\prime \prime}= Q_{\Phi,\vartheta}(f)(z-z_0)~.
\end{array}
\label{eqNewCohenclass2}
\end{equation}
\end{remark}

On the contrary, in some cases $W_{\vartheta} f$ is a linear perturbation of the Wigner distribution according to the definition of \cite{Cordero}. Let us recall this definition. Consider matrix $A \in \operatorname*{Gl}(2n,\mathbb{R})$,
\begin{equation}
A= \left(
\begin{array}{c c}
A_{11} & A_{12}\\
A_{21} & A_{22}
\end{array}
\right)~,
\label{eqLinearPerturbation1}
\end{equation}
where $A_{jk}$ $(j,k=1,2)$ are real $n \times n$ matrices.

A \textit{linear perturbation of the Wigner distribution} is a bilinear time-frequency representation of the form;
\begin{equation}
\mathcal{B}_A (f,g) (x, \xi)=\int_{\mathbb{R}^n} f\left(A_{11}x+A_{12} y \right) \overline{g\left(A_{21}x+A_{22} y \right) } e^{-2 i \pi \xi \cdot y} dy~,
\label{eqLinearPerturbation2}
\end{equation}
for $f,g, \in L^2 (\mathbb{R}^n)$.

As usual, we shall write $\mathcal{B}_A f=\mathcal{B}_A (f,f)$.

These bilinear distributions belong to Cohen's class if and only if the matrix $A$ is of the form \cite{Cordero}
\begin{equation}
A=A_M = \left(
\begin{array}{c c}
I_n & M + (1/2) I_n\\
I_n & M - (1/2) I_n
\end{array}
\right)~, 
\label{eqLinearPerturbation3}
\end{equation}
where $M \in \mathbb{R}^{n \times n}$. In particular, we obtain the Wigner distribution if $M= 0$.

Let us now investigate under what circumstances $W_{\vartheta}f$ can be a linear perturbation of this form.

\begin{theorem}\label{TheoremLinearPerturbation}
The distribution $W_{\vartheta}f$ (\ref{eq56}) is a linear perturbation of the Wigner function of the form (\ref{eqLinearPerturbation2}) if and only if the symplectic matrices $S^{(1)},~S^{(2)}$ are of the form:
\begin{equation}
S^{(1)}= \left(
\begin{array}{c c}
A_{11}^{-1} & 0\\
&\\
C^{(1)} & A_{11}^T
\end{array}
\right)~, \hspace{1 cm}  S^{(2)}= \left(
\begin{array}{c c}
0 & -2 A_{22}^T\\
&\\
\frac{1}{2} A_{22}^{-1} & D^{(2)}
\end{array}
\right)~,
\label{eqLinearPerturbation3A}
\end{equation}
where $A_{11}, A_{22} \in \operatorname*{Gl}(n,\mathbb{R})$ and the matrices $C^{(1)},~D^{(2)} \in \mathbb{R}^{n \times n}$ are such that $\left(A_{11}^{-1}\right)^T C^{(1)}$ and $A_{22} D^{(2)} $ are symmetric.

In that case the linear perturbation is given by: 
\begin{equation}
A=\left(
\begin{array}{c c}
A_{11} & -A_{22}\\
A_{11} & A_{22}
\end{array}
\right)~.
\label{eqLinearPerturbation6}
\end{equation}

\end{theorem}

\begin{proof}
From Proposition 3.3 of \cite{Cordero} the distribution $\mathcal{B}_A f$ is a real function if and only if:
\begin{equation}
A_{11}= A_{21}~,  \hspace{1 cm} A_{12}=-A_{22}~.
\label{eqLinearPerturbation4}
\end{equation}
Since $W_{\vartheta} f$ is a real function, then it can only be a linear perturbation, 
\begin{equation}
W_{\vartheta} f (x,\xi)=N \mathcal{B}_A f(x,\xi)~, 
\label{eqLinearPerturbation5}
\end{equation}
if the matrix $A$ is of the form (\ref{eqLinearPerturbation6}).

For $A$ to be nonsingular, then both $A_{11}$ and $A_{22}$ must be nonsingular.\footnote{In the terminology of \cite{Cordero}, the matrix $A$ is both a \textit{left- and a right-regular matrix}.} The constant $N \in \mathbb{R}$ was introduced in (\ref{eqLinearPerturbation5}) just to ensure that $W_{\vartheta}f$ is normalized.

Plugging (\ref{eqLinearPerturbation6}) in (\ref{eqLinearPerturbation2}) and performing the substitution $y=- 1/2 A_{22}^{-1} \tau$ yields: 
\begin{equation}
\begin{array}{c}
\mathcal{B}_A f (x, \xi)=\int_{\mathbb{R}^n} f\left(A_{11}x-A_{22} y \right) \overline{f\left(A_{11}x+A_{22} y \right) } e^{-2 i \pi \xi \cdot y} dy=\\
\\
=\frac{1}{2^n | \det A_{22}|} \int_{\mathbb{R}^n} f\left(A_{11}x + \frac{\tau }{2} \right) \overline{f\left(A_{11}x-\frac{\tau }{2}\right) } e^{-2 i \pi \left(-\frac{1}{2}(A_{22}^{-1})^T \xi \right) \cdot \tau} d \tau=\\
\\
= \frac{1}{2^n | \det A_{22}|} W_{\sigma} f \left(A_{11} x, -\frac{1}{2}(A_{22}^{-1})^T \xi  \right)~. 
\end{array}
\label{eqLinearPerturbation7}
\end{equation}
Thus (\ref{eqLinearPerturbation5}) can hold if and only if (cf.(\ref{eq56})):
\begin{equation}
\left(D^{(1,2)}\right)^{-1}= \left(
\begin{array}{c c}
A_{11} & 0\\
0 & -\frac{1}{2}(A_{22}^{-1})^T
\end{array}
\right) \Leftrightarrow D^{(1,2)}= \left(
\begin{array}{c c}
A_{11}^{-1} & 0\\
0 & - 2A_{22}^T
\end{array}
\right)~.
\label{eqLinearPerturbation8}
\end{equation}
By comparison with (\ref{eq35}), we conclude that:
\begin{equation}
\begin{array}{l c l}
A^{(1)}= A_{11}^{-1} & \hspace{0.5 cm}& B^{(1)}= 0\\
& & \\
A^{(2)}= 0 & \hspace{0.5 cm}& B^{(2)}=- 2A_{22}^T
\end{array}
\label{eqLinearPerturbation9}
\end{equation}
Consequently:
\begin{equation}
S^{(1)}=\left(
\begin{array}{c c}
A_{11}^{-1} & 0\\
&\\
C^{(1)} & D^{(1)}
\end{array}
\right)~, \hspace{1 cm} S^{(2)}=\left(
\begin{array}{c c}
0 & - 2A_{22}^T\\
&\\
C^{(2)} & D^{(2)}
\end{array}
\right)~.
\label{eqLinearPerturbation10}
\end{equation}
From (\ref{eqSymp3.1},\ref{eqSymp3.2}) $S^{(1)}$ is a symplectic matrices if and only if:
\begin{equation}
\begin{array}{l}
\left(A_{11}^{-1}\right)^T C^{(1)} ~\text{ is symmetric ,}~\\
\\
\left(A_{11}^{-1}\right)^T D^{(1)}= I_n~.
\end{array}
\label{eqLinearPerturbation11}
\end{equation}
Likewise, $S^{(2)}$ is a symplectic matrices if and only if:
\begin{equation}
\begin{array}{l}
A_{22} D^{(2)}~\text{ is symmetric ,}~\\
\\
2 A_{22} C^{(2)}=  I_n~,
\end{array}
\label{eqLinearPerturbation12}
\end{equation}
which concludes the proof.
\end{proof}

\begin{remark}
Notice that the matrix $S^{(1)}$ in (\ref{eqLinearPerturbation3A}) can be written as the product 
\begin{equation}
S^{(1)} =V_P M_L~,
\label{eqLinearPerturbation13}
\end{equation}
of the two symplectic matrices:
\begin{equation}
V_P= \left(
\begin{array}{c c}
I_n & 0\\
&\\
-P & I_n
\end{array}
\right)~, \hspace{1 cm} M_L=\left(
\begin{array}{c c}
L^{-1} & 0\\
&\\
0 & L^T
\end{array}
\right)~,
\label{eqLinearPerturbation14}
\end{equation}
where $P=-C^{(1)} A_{11}$ is symmetric and $L=A_{11}$ is nonsingular. Matrices of the form $V_P$ are called "symplectic shears" and those of the form $M_L$ are known as a "symplectic squeezing" \cite{Gosson}.

In the same vein, the matrix $S^{(2)}$ in (\ref{eqLinearPerturbation3A}) can be written as the product 
\begin{equation}
S^{(2)} =V_P M_L J~,
\label{eqLinearPerturbation13}
\end{equation}
where $J$ is the standard symplectic matrix, and $V_P,~M_L$ are given by (\ref{eqLinearPerturbation14}), where this time $L= - \frac{1}{2} \left(A_{22}^T\right)^{-1}$ and $P=2 D^{(2)} A_{22}^T$.
\end{remark}

\begin{remark}

The linear perturbations of the Wigner distributions (\ref{eqLinearPerturbation2}) form a subclass of the so-called metaplectic Wigner distributions. The latter are the time-frequency representations  of the form:
\begin{equation}\label{MetaplecticWD}
W_S(f,g)=\widehat S (f\otimes \overline{g})
\end{equation}
where $\widehat S$ is a metaplectic operator and $S\in  Sp(4n,\mathbb{R})$ is the associated symplectic matrix. These objects generalise the standard cross Wigner distribution \cite{Dias6,Dias7} and the short-time Fourier transform \cite{Cordero2} which can be obtained from (\ref{MetaplecticWD}) for particular metaplectic operators $\widehat S_W$ and $\widehat S_{ST}$, respectively\footnote{In \cite{Dias6,Dias7} the formula (\ref{MetaplecticWD}) is written using $\widehat g$ instead of $g$. This is just an alternative convention as the extra Fourier transform can be absorbed by $\widehat S$.}. They were initially introduced as intertwining operators $f \to W_S(f,g)$ for a general class of so-called dimensional extensions of pseudo-differential operators \cite{Dias8}; and used in \cite{Dias7} to obtain the metaplectic formulation of the Wigner function. 
More recently, the properties of the distributions (\ref{MetaplecticWD}) and particularly their relation to modulation spaces have been thoroughly studied in a series of papers by E. Cordero and collaborators (see \cite{Cordero2,Cordero3,Cordero4} and references therein).       

In the case of a non-standard symplectic structure $\Omega$, it seems natural to consider the following generalisation of (\ref{MetaplecticWD}):
\begin{equation}\label{NCMetaplecticWD}
W_{D,S}(f,g)=M_D^{-1} \widehat S (f\otimes \overline{g})
\end{equation} 
where $M_D U(z)=\left|\det D\right| U(Dz)$ and $D$ is the Darboux matrix such that $\Omega=DJ D^T$. These maps have been studied for the case $\widehat S=\widehat S_W$ in connection to their role as intertwining operators for a certain class of dimensionally extended Weyl pseudo-differential operators with symbols defined on a non-standard symplectic space \cite{Dias4}.  
They are interesting objects generalising both (\ref{MetaplecticWD}) and the Wigner functions $W_\vartheta f$ associated with the LCTs (\ref{eq56}) (which can be written in the form (\ref{NCMetaplecticWD}) for $D$ given by (\ref{eq35}) and $\widehat S=\widehat S_W$). An interesting topic, which we leave for future research, is the relation between the representations (\ref{MetaplecticWD}) and (\ref{NCMetaplecticWD}), and in particular whether the Wigner functions $W_\vartheta f$ can be rewritten in the form (\ref{MetaplecticWD}) for a suitable choice of the metaplectic operator. 
\end{remark}

\section{Hardy's uncertainty principle for metaplectic operators}

To prove Hardy's theorem \cite{Hardy} in the LCT space, we shall need some preliminary results. In the sequel $A,B,C,D$ denote the $n \times n $ matrices of the decomposition (\ref{eqSymp3.1}) of $S \in \operatorname*{Sp}(2n,\mathbb{R})$, where we assume as previously that $\det B \neq 0$.

\begin{proposition}\label{PropositionHardy1}
Suppose that there exists $C>0$, such that $|f(x)| \leq C e^{-Mx^2}$, for all $x \in \mathbb{R}^n$, where $M$ is a real, symmetric, positive-definite $n \times n$ matrix. Then $\widehat{f}_S$ is entire on $\mathbb{C}^n$.
\end{proposition}

\begin{proof}
We have for $z=\xi + i \eta \in \mathbb{C}^n$:
\begin{equation}
\begin{array}{c}
|\widehat{f}_S(z)| \leq \frac{1}{\sqrt{| \det B |}} \left|\int_{\mathbb{R}^n} f(x) e^{i \pi z \cdot DB^{-1} z + i \pi x \cdot B^{-1} A x -2 i \pi x \cdot B^{-1} z} dx\right| \\
\\
\leq \frac{1}{\sqrt{| \det B |}}\int_{\mathbb{R}^n}  e^{-x \cdot M x + 2 \pi x \cdot B^{-1}\eta -2 \pi \xi \cdot \left(DB^{-1}\right)_+ \eta } dx < \infty~,
\end{array}
\label{eqHardy2}
\end{equation}
for all $\eta, \xi \in \mathbb{R}^n$, and $(DB^{-1})_+=\frac{1}{2}\left(DB^{-1}+ (DB^{-1})^T \right)$ denotes the symmetric part of $DB^{-1}$. Consequently, we can differentiate under the integral sign.
\end{proof}

Before we proceed let us recall the Phragm\'en-Lindel\"of principle for completeness \cite{Phragmen}:

\begin{theorem}[Phragm\'en-Lindel\"of]\label{TheoremPhragmenLindelof}
Let $\Omega$ be the section of the complex plane:
\begin{equation}
\Omega = \left\{z=Re^{i \theta}:~0 \leq \alpha < \theta < \beta \right\}~,
\label{eqHardy3}
\end{equation}
such that $\beta -\alpha < 2 \pi$. Let $f$ be holomorphic in $\Omega$ and continuous in $\partial \Omega$. If there exist constants $M, c, \tau>0$, $0 < \rho < \frac{\pi}{\beta- \alpha}$, such that $|f(z)| \leq M$ on $\partial \Omega$ and $|f(z)| \leq c e^{ \tau |z|^{\rho}}$ on $\Omega$, then $|f(z)| \leq M$ on all of $\Omega$. 
\end{theorem}

As a final result, we shall also consider the LCT of a generalized Gaussian in $n=1$.

\begin{lemma}\label{LemmaHardy1}
Let 
\begin{equation}
f(x)= e^{- \left(\alpha + i \pi a/b\right)x^2}~,
\label{eqHardy4}
\end{equation}
for $x \in \mathbb{R}$, $\alpha>0$ and, as before,
\begin{equation}
S = \left(
\begin{array}{c c}
a & b\\
c & d
\end{array}
\right) \in \operatorname*{Sp}(2,\mathbb{R})~,
\label{eqHardy5}
\end{equation}
with $b \neq 0$.

Then:
\begin{equation}
\left(\mathcal{L}_S f\right)(\xi)= \widehat{f}_S (\xi)= C e^{- K \xi^2} ~,
\label{eqHardy6}
\end{equation}
where $C \in \mathbb{C}$ is a constant, and 
\begin{equation}
K=\frac{\pi^2}{b^2 \alpha} - \frac{i \pi d}{b}~.
\label{eqHardy7}
\end{equation}
\end{lemma}

\begin{proof}
We have:
\begin{equation}
\begin{array}{c}
\widehat{f}_S(\xi)= \frac{1}{\sqrt{ib}} \int_{\mathbb{R}} \exp \left[-\left(\alpha + \frac{i \pi a}{b} \right) x^2 + \frac{i \pi d}{b} \xi^2 + \frac{i \pi a}{b} x^2 - \frac{2 i \pi }{b} x \xi \right] dx = \\
\\
= \frac{e^{\frac{i\pi d}{b} \xi^2}}{\sqrt{ib}} \int_{\mathbb{R}} e^{- \alpha x^2 - \frac{2 i \pi }{b} x \xi } dx =\frac{e^{\frac{i\pi d}{b} \xi^2}}{\sqrt{ib}} \sqrt{\frac{\pi}{\alpha}} e^{- \frac{\pi^2 \xi^2}{\alpha b^2}}=\\
\\
= \sqrt{\frac{\pi}{i \alpha b}} \exp \left[- \left(\frac{\pi^2}{\alpha b^2}- \frac{i \pi d}{b} \right) \xi^2 \right]~,
\end{array}
\label{eqHardy8}
\end{equation}
which proves the result.
\end{proof}

\subsection{Hardy's uncertainty principle for metaplectic operators in one dimension}

We are now in a position to state and prove Hardy's uncertainty principle in the LCT space. We start with the one-dimensional case.

\begin{theorem}\label{TheoremHardyLCT2}
Let $f: \mathbb{R} \to \mathbb{C}$ satisfy:
\begin{eqnarray}
|f(x)| \leq C_1 e^{- \alpha x^2} \label{eqHardy9A}\\
|\widehat{f}_S(\xi)| \leq C_2 e^{- \beta \xi^2} \label{eqHardy9B}
\end{eqnarray}
for some $C_1,C_2, \alpha, \beta >0$, and $S$ as in (\ref{eqHardy5}).

\begin{enumerate}
\item[(a)] If $\alpha \beta = \frac{\pi^2}{b^2}$, then there exists $A \in \mathbb{C}$, such that:
\begin{equation}
f(x) =A e^{- \left(\alpha + i \pi a/b\right)x^2}~.
\label{eqHardy10}
\end{equation}

\item[(b)] If $\alpha \beta > \frac{\pi^2}{b^2}$, then $f \equiv 0$.
\end{enumerate}
\end{theorem}

\begin{proof}
Notice that, in view of (\ref{eqHardy9A},\ref{eqHardy9B}), $f,\widehat{f}_S \in L^1(\mathbb{R})$, and by the Riemann-Lebesgue Theorem they are both continuous.

We start by assuming that (a) is valid and prove (b).

Suppose that $f$ is not identically zero and $\alpha \beta > \frac{\pi^2}{b^2} \Leftrightarrow \beta >\frac{\pi^2}{\alpha b^2}$. In that case, we have:
\begin{equation}
\begin{array}{l}
|f(x)| \leq C_1 e^{- \alpha x^2} \\
|\widehat{f}_S(\xi)| \leq C_2 e^{- \beta \xi^2} \leq C_2 e^{- \frac{\pi^2}{\alpha b^2} \xi^2}~. 
\end{array}
\label{eqHardy11}
\end{equation}
In particular, the second inequality is strict for $\xi \neq 0$. But, since $\alpha \times \frac{\pi^2}{\alpha b^2}= \frac{\pi^2}{b^2}$, it follows from (a) that:
\begin{equation}
f(x)=A e^{-\left(\alpha + i \pi a/b\right)x^2} ~,
\label{eqHardy12}
\end{equation}
for some $A \in \mathbb{C}$. But from Lemma \ref{LemmaHardy1}, there exists $C^{\prime} >0$, such that:
\begin{equation}
|\widehat{f}_S (\xi)| = C^{\prime} e^{- \frac{\pi^2}{b^2 \alpha} \xi^2} ~.
\label{eqHardy13}
\end{equation}
Since, by assumption, $\frac{\pi^2}{b^2 \alpha} < \beta$, this contradicts (\ref{eqHardy9B}).

Next we prove (a). Assume that $\alpha \beta = \frac{\pi^2}{b^2}$. From (\ref{eqHardy9A}), we have:
\begin{equation}
\begin{array}{c}
\int_{\mathbb{R}} |f(x)| e^{\frac{2 \pi}{b} \eta (x-d \xi)} dx \leq C_1 \int_{\mathbb{R}} e^{- \alpha x^2 + \frac{2\pi}{b} \eta x - \frac{2 \pi d}{b} \eta \xi} dx\\
\\
 \leq C_1 \sqrt{\frac{\pi}{\alpha}} e^{- \frac{2 \pi d}{b} \eta \xi + \frac{\pi^2}{\alpha b^2} \eta^2}  < \infty~,
\end{array}
\label{eqHardy14}
\end{equation}
for all $\eta , \xi \in \mathbb{R}$.

By Proposition \ref{PropositionHardy1}, $\widehat{f}_S (z)$ is entire. We also see from (\ref{eqHardy2}) that (\ref{eqHardy9A}) implies:
\begin{equation}
|\widehat{f}_S (z)| \lesssim \int_{\mathbb{R}} e^{- \alpha x^2 - \frac{2 \pi d}{b} \xi \eta + \frac{2 \pi}{b} x \eta} dx \lesssim e^{\frac{\pi^2 \eta^2}{\alpha b^2} - \frac{2 \pi d}{b} \eta \xi}~,
\label{eqHardy15}
\end{equation}
for all $z= \xi + i \eta \in \mathbb{C}$.

From (\ref{eqHardy9B}) we have:
\begin{equation}
|\widehat{f}_S (\xi)| \lesssim e^{- \frac{\pi^2 \xi^2}{\alpha b^2}}~,
\label{eqHardy16}
\end{equation}
for all $\xi \in \mathbb{R}$.

Let $F_{\sigma} (z)= e^{{\cal T}_{\sigma} z^2} \widehat{f}_S(z)$, where ${\cal T}_{\sigma}=T_1+ i T_2$, with $T_1= \frac{\pi^2}{\alpha b^2}$, $T_2= - \frac{\pi d}{b} + \sigma$, and where $\sigma>0$ is for now arbitrary.

The function $F_{\sigma} (z)$ is entire. Moreover, we have:
\begin{equation}
|F_{\sigma} (z)|= \left|e^{(T_1+iT_2)(\xi^2- \eta^2+ 2 i \xi \eta)} \widehat{f}_S(z) \right| =e^{T_1 (\xi^2-\eta^2)- 2 T_2 \xi \eta} |\widehat{f}_S(z)|~.
\label{eqHardy17}
\end{equation}
On the real axis $(\eta=0)$, we have:
\begin{equation}
|F_{\sigma} (\xi)|=e^{T_1 \xi^2} |\widehat{f}_S (\xi)| \leq M_1~,
\label{eqHardy18}
\end{equation}
where we used (\ref{eqHardy16}), and $M_1>0$ is a constant which is independent of $\sigma$.

On the other hand:
\begin{equation}
\begin{array}{c}
|F_{\sigma} (z)|=e^{T_1 (\xi^2-\eta^2)- 2 T_2 \xi \eta} |\widehat{f}_S(z)|\\
\\
\lesssim e^{T_1(\xi^2-\eta^2)- 2 T_2 \xi \eta + \frac{\pi^2 \eta^2}{\alpha b^2}- \frac{2 \pi d}{b} \xi \eta} \lesssim e^{\frac{\pi^2 \xi^2}{\alpha b^2} - 2 \sigma \xi \eta}~,
\end{array}
\label{eqHardy19}
\end{equation}
where we used (\ref{eqHardy15}).

Setting $\xi = R \cos(\theta)$, $\eta = R \sin (\theta)$:
\begin{equation}
|F_{\sigma} (z)|\lesssim e^{\frac{\pi^2 }{\alpha b^2}R^2 \cos^2 (\theta) - 2 \sigma R^2 \sin(\theta) \cos(\theta)}~.
\label{eqHardy20}
\end{equation}
It is straightforward to show that, for fixed $\sigma>0$, if we choose $\theta=\theta_0 < \frac{\pi}{2}$ sufficiently close to $\frac{\pi}{2}$, the exponent is negative:
\[
\frac{\pi^2 }{\alpha b^2}R^2 \cos^2 (\theta_0) - 2 \sigma R^2 \sin(\theta_0) \cos(\theta_0)=-2 \sigma R^2 \left(\pi/2-\theta_0\right) + \mathcal{O} \left(\pi /2 - \theta_0\right)^2.
\]
Hence:
\begin{equation}
|F_{\sigma} (z)|\leq M_2~,
\label{eqHardy21}
\end{equation}
on the half line $R>0$, $\theta=\theta_0$, and where $M_2 >0$ is a constant independent of $\sigma$. If we set $M=\text{max} \left\{M_1,M_2 \right\}$, we can assume the same constant in (\ref{eqHardy18}) and (\ref{eqHardy21}).

Since $\theta_0< \frac{\pi}{2}$, we have $\frac{\pi}{\theta_0} >2$, and hence we may choose $\rho=2$ in the Phragm\'en-Lindel\"of principle to conclude that:
\begin{equation}
|F_{\sigma} (z)|\leq M~,
\label{eqHardy22}
\end{equation}
for all $z$ in the section $\Omega=\left\{z=R e^{i \theta}~, ~0 < \theta < \theta_0 < \frac{\pi}{2} \right\}$ of the complex plane.

If we let $\theta_0 \to \frac{\pi}{2}^-$ and subsequently $\sigma \to 0^+$, we conclude that the function 
\[
G(z)=e^{\left(\frac{\pi^2}{\alpha b^2}- \frac{i \pi d}{b} \right)z^2} \widehat{f}_S(z)
\]
satisfies $|G(z)|\leq M$ on the first quadrant. By following the same arguments, we can prove that the same result holds in the remaining quadrants, and thus $|G(z)|\leq M$ everywhere in $\mathbb{C}$. 

By Liouville's Theorem, we conclude that there exists $A \in \mathbb{C}$ such that $G(z)=A \Rightarrow \widehat{f}_S(\xi )=A e^{- \left(\frac{\pi^2}{\alpha b^2}- \frac{i \pi d}{b} \right) \xi^2}$. By Lemma \ref{LemmaHardy1} and the fact that $\mathcal{L}_S$ is injective the result follows.
\end{proof}

\begin{corollary}\label{CorollaryHardy1}
Let $f: \mathbb{R} \to \mathbb{C}$ satisfy:
\begin{eqnarray}
|(\widehat{S^{(1)}}f)(x)| \leq C_1 e^{- \alpha x^2} \label{eqHardy23A}\\
|(\widehat{S^{(2)}}f)(\xi)| \leq C_2 e^{- \beta \xi^2} \label{eqHardy23B}
\end{eqnarray}
for some $C_1,C_2, \alpha, \beta >0$, and 
\begin{equation}
S^{(j)} = \left(
\begin{array}{c c}
a_j & b_j\\
c_j & d_j
\end{array}
\right) \in \operatorname*{Sp}(2,\mathbb{R})
\label{eqHardy24}
\end{equation}
for $j=1,2$, such that:
\begin{equation}
a_1b_2-a_2b_1 \neq 0~.
\label{eqHardy24A}
\end{equation}

\begin{enumerate}
\item If $\alpha \beta = \frac{\pi^2}{(a_1b_2-a_2b_1)^2}$, then there exists $A \in \mathbb{C}$, such that:
\begin{equation}
f(x) =A \exp \left[- \left(\alpha + i \pi \frac{a_2d_1-b_2c_1}{a_1b_2-a_2b_1}\right)x^2 \right]~.
\label{eqHardy25}
\end{equation}

\item If $\alpha \beta > \frac{\pi^2}{(a_1b_2-a_2b_1)^2}$, then $f \equiv 0$.
\end{enumerate}
\end{corollary}

\begin{proof}
Let $g(x)=(\widehat{S^{(1)}}f)(x)$, and define $S=S^{(2)}(S^{(1)})^{-1} \in \operatorname*{Sp}(2,\mathbb{R})$. Then:
\begin{equation}
(\widehat{S}g)(\xi)= \left(\widehat{S^{(2)}} (\widehat{S^{(1)}})^{-1} \widehat{S^{(1)}} f\right)(\xi)=\left(\widehat{S^{(2)}} f\right)(\xi)~.
\label{eqHardy26}
\end{equation}
Now notice that
\begin{equation}
S=S_2S_1^{-1}= \left(
\begin{array}{c c}
a_2d_1-b_2c_1 & -a_2b_1+b_2a_1\\
c_2d_1-d_2c_1 & -c_2 b_1+d_2a_1
\end{array}
\right)~.
\label{eqHardy27}
\end{equation}
Hence, from (\ref{eqHardy24A}), $S$ is a free symplectic matrix.

If we apply the previous theorem to $g$ and $\widehat{S}g= \widehat{g}_S$ the result follows.
\end{proof}

\subsection{Multi-dimensional Hardy uncertainty principle for the LCT}

To proceed to the higher-dimensional case, we need some preliminary results. We shall follow closely the techniques developed in \cite{GoLu}.

\begin{lemma}\label{LemmaHardy2}
Given $S  \in \operatorname*{Sp}(2n,\mathbb{R})$, let $S_L=M_LSM_L^T$, where
\begin{equation}
M_L=\left(
\begin{array}{c c}
L^T &0\\
0 & L^{-1}
\end{array}
\right)~,
\label{eqHardy28}
\end{equation}
with $L \in  \operatorname*{Gl}(n,\mathbb{R})$.  

Then:
\begin{equation}
\left(\mathcal{L}_{S_L} f_L\right)(\xi)= \frac{1}{|\det L|} \left(\mathcal{L}_S f \right) \left((L^{-1})^T \xi \right)~,
\label{eqHardy29}
\end{equation}
where $f_L(x)=f(Lx)$.
\end{lemma}

\begin{proof}
If we use the decomposition (\ref{eqSymp3.1}) for $S$, we have:
\begin{equation}
S_L=\left(
\begin{array}{c c}
L^TAL & L^TB(L^{-1})^T\\
L^{-1}CL & L^{-1} D(L^{-1})^T
\end{array}
\right)~.
\label{eqHardy29}
\end{equation}
Consequently:
\begin{equation}
\begin{array}{c}
\left(\mathcal{L}_{S_L} f_L\right)(\xi)=\frac{1}{i^{n/2}\sqrt{\det B}} \int_{\mathbb{R}^n} f(Lx)
\exp \left[ i \pi \left(\xi \cdot L^{-1}DB^{-1}(L^{-1})^T \xi + x \cdot L^T B^{-1}AL x \right)\right.\\
\\
\left.-2 i \pi x \cdot L^T B^{-1} (L^T)^{-1} \xi \right] dx=\\
\\
= \frac{|\det L^{-1}|}{i^{n/2}\sqrt{\det B}} \int_{\mathbb{R}^n} f(y)
\exp \left[ i \pi \left((L^{-1})^T\xi \cdot DB^{-1}(L^{-1})^T \xi + y \cdot  B^{-1}Ay\right)\right.\\
\\
\left.-2 i \pi y \cdot  B^{-1} (L^T)^{-1} \xi \right] dy =\frac{1}{|\det L|} \left(\mathcal{L}_S f \right) \left((L^{-1})^T \xi \right)~, 
\end{array}
\label{eqHardy30}
\end{equation}
which proves the result.
\end{proof}

Before we proceed let us make the following observation. If $M,N$ are positive-definite $n \times n$ matrices, then the eigenvalues of $MN$ are real and positive, because $MN$ has the same eigenvalues as the symmetric and positive-definite matrices $M^{1/2}N M^{1/2}$ and $N^{1/2} M N^{1/2}$.  

\begin{lemma}\label{LemmaHardy3}
Let $M,N$ be positive-definite $n \times n$ matrices. There exists $L \in  \operatorname*{Gl}(n,\mathbb{R})$ such that
\begin{equation}
L^TML=L^{-1}N(L^{-1})^T= \Lambda~,
\label{eqHardy31}
\end{equation}
where $\Lambda= \text{diag}(\sqrt{\lambda_1}, \cdots, \sqrt{\lambda_n})$ is the diagonal matrix whose eigenvalues are the square roots of the eigenvalues $\lambda_1, \cdots, \lambda_n$ of $MN$.
\end{lemma}

\begin{proof}
See Lemma 1 of \cite{GoLu}.
\end{proof}

\begin{lemma}\label{LemmaHardy4}
Let $n>1$. For $1 \leq j \leq n$ let $f_j$ be a continuous complex-valued function of $(x_1, \cdots, \widetilde{x}_j, \cdots, x_n) \in \mathbb{R}^{n-1}$ (the tilde suppressing the term it covers), and $g_j$ a continuous complex-valued function of $x_j \in \mathbb{R}$. If
\begin{equation}
h=f_1 \otimes g_1= \cdots= f_n \otimes g_n~,
\label{eqHardy32}
\end{equation}
then there exists a constant $C \in \mathbb{C}$, such that
\begin{equation}
h=C \left(g_1 \otimes \cdots \otimes g_n \right)~.
\label{eqHardy33}
\end{equation}
Here $\left(f \otimes g\right)(y_1,y_2)= f(y_1) g(y_2)$ denotes the tensor product.
\end{lemma}

\begin{proof}
See Lemma 4 of \cite{GoLu}.
\end{proof}

We can now prove the multidimensional Hardy uncertainty principle in the LCT space.

\begin{theorem}\label{TheoremMultidimensionalHardy}
Let $M,N$ be two real symmetric, positive-definite $n \times n$ matrices and $f \in L^2 (\mathbb{R}^n) \backslash \left\{0 \right\}$. Assume that
\begin{eqnarray}
|f(x)| \lesssim e^{-Mx^2} \label{eqHardy34A}\\
|\widehat{f}_S (\xi)| \lesssim e^{-N \xi^2} \label{eqHardy34B}
\end{eqnarray}
where 
\begin{equation}
S= \left(
\begin{array}{c c}
A & B\\
C & D
\end{array}
\right)  \in  \operatorname*{Sp}(2n,\mathbb{R})~
\label{eqHardy35}
\end{equation}
is a free symplectic matrix.

Let $\lambda_1, \cdots, \lambda_n$ denote the eigenvalues of $MB^TNB$.

\begin{enumerate}
\item[(a)] The eigenvalues of $MB^TNB$ satisfy: $\lambda_j \leq \pi^2$, $j=1, \cdots, n$.

\item[(b)] If $\lambda_1= \cdots= \lambda_n=\pi^2$, then there exists $K \in \mathbb{C}$ such that
\begin{equation}
f(x)=K e^{-x \cdot\left(M+i \pi B^{-1}A \right)x} ~.
\label{eqHardy36}
\end{equation}
\end{enumerate}
\end{theorem}

\begin{proof}
Notice that from(\ref{eqHardy34A},\ref{eqHardy34B}), we have $f,\widehat{f}_S \in L^p (\mathbb{R}^n)$ for all $p \in \left[1, \infty\right]$, and we may safely apply Fubini's theorem in all integrations that follow.

\begin{enumerate}
\item[(a)] Let $L \in \operatorname*{Gl}(n,\mathbb{R})$ be the matrix such that (see Lemma \ref{LemmaHardy3}) such that:
\begin{equation}
L^TML=L^{-1}B^TNB (L^{-1})^T= \Lambda ~,
\label{eqHardy37}
\end{equation}
where $\Lambda= \text{diag}(\sqrt{\lambda_1},\cdots, \sqrt{\lambda_n})$, and $\lambda_1, \cdots, \lambda_n$ are the eigenvalues of $MB^TNB$, which we assume, without loss of generality, to be ordered decreasingly: $\lambda_1 \geq \lambda_2 \geq \cdots \geq \lambda_n$. It then suffices to prove the result for $\lambda_1$.

In the sequel we shall consider the function $g(x)$ defined by:
\begin{equation}
f(x)= g(x)e^{-i \pi x \cdot B^{-1}Ax}~.
\label{eqfunctiong1}
\end{equation}
From (\ref{eqHardy34A}) we have:
\begin{equation}
|g(x)|=|f(x)| \lesssim e^{-Mx^2}~.
\label{eqfunctiong2}
\end{equation}
On the other hand:
\begin{equation}
\begin{array}{c}
\widehat{f}_S(\xi)= \frac{1}{i^{n/2}\sqrt{\det (B)}} \int_{\mathbb{R}^n} f(x)  e^{i \pi \xi \cdot DB^{-1} \xi +i \pi x \cdot B^{-1}A x - 2 i \pi x \cdot B^{-1} \xi } dx=\\
\\
= \frac{1}{i^{n/2}\sqrt{\det (B)}} \int_{\mathbb{R}^n} g(x)  e^{i \pi \xi \cdot DB^{-1} \xi - 2 i \pi x \cdot B^{-1} \xi } dx= \widehat{g}_{S^{\prime}}(\xi)~,
\end{array}
\label{eqfunctiong3}
\end{equation}
where
\begin{equation}
S^{\prime}=\left(
\begin{array}{c c}
0 & B\\
-(B^{-1})^T & D
\end{array}
\right)~.
\label{eqfunctiong4}
\end{equation}
Since, by assumption, $S \in \operatorname*{Sp}(2n,\mathbb{R})$, we have that $B^TD$ is symmetric (cf.(\ref{eqSymp3.2})), and therefore so is $DB^{-1}$. Consequently, $S^{\prime}$ is also a free symplectic matrix.

It follows from (\ref{eqHardy34B},\ref{eqfunctiong3}) that:
\begin{equation}
|\widehat{g}_{S^{\prime}} (\xi)|
=|\widehat{f}_S(\xi)| \lesssim e^{-N \xi^2}~.
\label{eqfunctiong5}
\end{equation}
As before, let $g_L(x)=g(Lx)$. We thus have from (\ref{eqHardy37},\ref{eqfunctiong2}):
\begin{equation}
|g_L(x)|\lesssim e^{-x \cdot L^T MLx}= e^{- \Lambda x^2}~.
\label{eqHardy38}
\end{equation}
Next we make the following remark. Let
\begin{equation}
S_B^{\prime}= \left(
\begin{array}{c c}
B^{-1} &0\\
0 & B^T
\end{array}
\right) \left(
\begin{array}{c c}
0 & B\\
-(B^{-1})^T & D
\end{array}
\right)=\left(
\begin{array}{c c}
0 & I_n\\
-I_n & B^TD
\end{array}
\right) ~.
\label{eqHardy39}
\end{equation}
Then $S_B^{\prime}$ is also a free symplectic matrix. We thus have:
\begin{equation}
\begin{array}{c}
\widehat{g}_{S^{\prime}} (B \xi)= \frac{1}{i^{n/2} \sqrt{\det B}} \int_{\mathbb{R}^n} g(x) 
e^{i \pi \xi \cdot B^T D \xi-2i \pi x \cdot \xi } dx= \\
\\
= \frac{1}{\sqrt{\det B}}\widehat{g}_{S_B^{\prime}} (\xi)~. 
\end{array}
\label{eqHardy40}
\end{equation}
Thus:
\begin{equation}
|\widehat{g}_{S_B^{\prime}} (\xi)|= \left|\sqrt{\det B} \widehat{g}_{S^{\prime}} (B \xi) \right| \lesssim e^{- \xi \cdot B^TNB \xi} = e^{-N_B \xi^2}~,
\label{eqHardy41}
\end{equation}
where
\begin{equation}
N_B=B^TNB~.
\label{eqHardy42}
\end{equation}
Let $S_L$ be as in Lemma \ref{LemmaHardy2}:
\begin{equation}
S_L^{\prime}=M_LS^{\prime}M_L^T=\left(
\begin{array}{c c}
0 & L^T B(L^{-1})^T\\
-L^{-1}(B^{-1})^T L & L^{-1}D(L^{-1})^T
\end{array}
\right)~.
\label{eqHardy43}
\end{equation}
Define also:
\begin{equation}
\begin{array}{c}
S_{L,B}^{\prime}=\left(
\begin{array}{c c}
L^TB^{-1}(L^{-1})^T & 0\\
0 & L^{-1} B^T L
\end{array}
\right)~ S_L^{\prime}= \\
\\
=\left(
\begin{array}{c c}
0 & I_n\\
-I_n & L^{-1} B^T D (L^{-1})^T
\end{array}
\right)~.
\end{array}
\label{eqHardy44}
\end{equation}
Then:
\begin{equation}
\begin{array}{c}
\left(\mathcal{L}_{S_{L,B}^{\prime}} g_L \right) (\xi)= \frac{1}{i^{n/2}} \int_{\mathbb{R}^n} g(Lx) e^{i \pi \xi \cdot L^{-1}B^T D (L^{-1})^T \xi - 2 i \pi x \cdot \xi } dx=\\
\\
=  \frac{1}{i^{n/2}|\det L|} \int_{\mathbb{R}^n} g(y) e^{ i \pi \left((L^{-1})^T\xi\right) \cdot B^T D (L^{-1})^T \xi - 2 i \pi y \cdot \left((L^{-1})^T \xi \right)} dy\\
\\
\Leftrightarrow \left(\mathcal{L}_{S_{L,B}^{\prime}} g_L \right) (\xi)= \frac{1}{|\det L|} \widehat{g}_{S_B^{\prime}} \left((L^{-1})^T \xi \right)~.
\end{array}
\label{eqHardy45}
\end{equation}
From (\ref{eqHardy37},\ref{eqHardy41},\ref{eqHardy45}):
\begin{equation}
\begin{array}{c}
\left|\left(\mathcal{L}_{S_{L,B}^{\prime}}g_L \right) (\xi) \right| = 
\frac{1}{|\det L|} \left|\widehat{g}_{S_B^{\prime}} \left((L^{-1})^T \xi \right) \right| \lesssim e^{-\xi \cdot L^{-1}N_B (L^{-1})^T \xi} \\
\\
\Leftrightarrow \left|\left(\mathcal{L}_{S_{L,B}^{\prime}}g_L \right) (\xi) \right|  \lesssim e^{- \Lambda \xi^2}~.
\end{array}
\label{eqHardy46}
\end{equation}
Now consider the function:
\begin{equation}
g_{L}^{\prime}(x_1)= g_L(x_1,x^{\prime})~,
\label{eqHardy47}
\end{equation}
for fixed $x^{\prime}=(x_2,\cdots,x_n) \in \mathbb{R}^{n-1}$. Since, by assumption, $g(x)$ is not identically zero, there must exist $x^{\prime} \in \mathbb{R}^{n-1}$, such that $g_L^{\prime}(x_1)$ is not identically zero. In the sequel, we shall assume that we have chosen such a $x^{\prime}$.

We have from (\ref{eqHardy38}):
\begin{equation}
|g_{L}^{\prime} (x_1)|  \leq c_1(x^{\prime}) e^{- \sqrt{\lambda_1} x_1^2}~,
\label{eqHardy48}
\end{equation}
where $c_1(x^{\prime})$ is a function of $x^{\prime}$ only.

On the other hand (cf. first line of (\ref{eqHardy45})):
\begin{equation}
\begin{array}{c}
\int_{\mathbb{R}}\cdots \int_{\mathbb{R}}  \left(\mathcal{L}_{S_{L,B}^{\prime}}g_L \right) (\xi) \exp \left[ - i \pi \xi \cdot L^{-1}B^TD(L^{-1})^T \xi+ \right.\\
\\
\left. +i \pi  \left(L^{-1}B^TD(L^{-1})^T\right)_{1,1}\xi_1^2 +2 i \pi \sum_{j=2}^n \xi_j x_j \right] d \xi_2 \cdots d \xi_n=\\
\\
=i^{-n/2} \int_{\mathbb{R}^n} \int_{\mathbb{R}}\cdots  \int_{\mathbb{R}}  g_L(y) \exp \left[  i \pi \left(L^{-1}B^TD(L^{-1})^T \right)_{1,1} \xi_1^2 + \right.\\
\\
\left. + 2 i \pi \sum_{j=2}^n \xi_j (x_j -y_j) -2 i \pi y_1  \xi_1 \right] d \xi_2 \cdots d \xi_n dy=\\
\\
=i^{-n/2} \int_{\mathbb{R}}   g_{L} (x_1,x^{\prime}) e^{ i \pi \left(L^{-1}B^TD(L^{-1})^T \right)_{1,1} \xi_1^2  
 -2 i \pi x_1  \xi_1 } d x_1= \\
 \\
 =\frac{1}{i^{(n-1)/2}} \left(\mathcal{L}_{S^{\prime (1)}}^{(1)} g_{L}^{\prime}\right) (\xi_1)~.
\end{array}
\label{eqHardy49}
\end{equation}
Here $\mathcal{L}_{S^{\prime (1)}}^{(1)}$ denotes the one-dimensional LCT with respect to the $2 \times 2$ free symplectic matrix:
\begin{equation}
S^{\prime (1)}= \left(
\begin{array}{c c}
0 & 1\\
-1 & \left(L^{-1}B^TD (L^{-1})^T \right)_{1,1}
\end{array}
\right)~.
\label{eqHardy50}
\end{equation}
It follows that:
\begin{equation}
\left|\left(\mathcal{L}_{S^{\prime (1)}}^{(1)} g_{L}^{\prime} \right) (\xi_1)\right| \lesssim \int_{\mathbb{R}} \cdots \int_{\mathbb{R}} \left|\left(\mathcal{L}_{S_{L,B}^{\prime}}g_L \right) (\xi)      \right| d\xi_2 \cdots d \xi_n\lesssim e^{- \sqrt{\lambda_1} \xi_1^2} ~.
\label{eqHardy51}
\end{equation}
Since $g_L^{\prime}$ is not identically zero, we conclude from (\ref{eqHardy48},\ref{eqHardy51}) and Theorem \ref{TheoremHardyLCT2} that:
\begin{equation}
\lambda_1 \leq \pi^2~.
\label{eqHardy52}
\end{equation}

\item[(b)] Now suppose that $\lambda_1=\lambda_2 = \cdots=\lambda_n=\pi^2$. We thus have from (\ref{eqHardy38},\ref{eqHardy46}):
\begin{eqnarray}
|g_L(x)| \lesssim e^{- \pi x^2} \hspace{1.4 cm} \label{eqHardy53A}\\
\left| \left(\mathcal{L}_{S_{L,B}^{\prime}} g_L\right) (\xi) \right| \lesssim e^{- \pi \xi^2}\label{eqHardy53B}
\end{eqnarray}
From (\ref{eqHardy48},\ref{eqHardy51}):
\begin{eqnarray}
|g_L^{\prime}(x_1)| \leq c_1(x^{\prime}) e^{- \pi x_1^2} \hspace{0.4 cm} \label{eqHardy54A}\\
\left| \left(\mathcal{L}_{S^{\prime(1)}}^{(1)} g_L^{\prime}\right) (\xi_1) \right| \lesssim e^{- \pi \xi_1^2}\label{eqHardy54B}
\end{eqnarray}
From (\ref{eqHardy54A},\ref{eqHardy54B}) and Hardy's Theorem (theorem \ref{TheoremHardyLCT2}):
\begin{equation}
g_L^{\prime} (x_1)= g_1(x^{\prime}) e^{ - \pi x_1^2 }~,
\label{eqHardy58}
\end{equation}
for some function $g_1: \mathbb{R}^{n-1} \to \mathbb{C}$ of $x^{\prime}$ only.

Applying the same argument to the remaining variables $x_2, \cdots, x_n$, we conclude that there exist functions $g_j$, $j=1, \cdots, n$, such that:
\begin{equation}
g_L(x)= g_j (x_1, \cdots, \widetilde{x_j}, \cdots, x_n) e^{ - \pi x_j^2 }~,
\label{eqHardy59}
\end{equation}
for all $j=1, \cdots, n$.

From Lemma \ref{LemmaHardy4}, there exists $K \in \mathbb{C}$ such that:
\begin{equation}
g(Lx)= K e^{- \pi x^2}~.
\label{eqHardy60}
\end{equation}
Since $\Lambda= \pi I_n \Leftrightarrow L^TML= \pi I_n \Leftrightarrow M=\pi (L^{-1})^T L^{-1}$, we conclude that:
\begin{equation}
g(x)=K e^{-Mx^2}~.
\label{eqHardy61}
\end{equation}
But from (\ref{eqfunctiong1}) it follows that:
\begin{equation}
f(x)=e^{-x \cdot(M+i \pi B~{-1}A)x}~,
\label{eqHardy60A}
\end{equation}
which concludes the proof.
\end{enumerate}
\end{proof}

We may generalize the previous result for arbitrary LCT's $\mathcal{L}_{S_1}$, $\mathcal{L}_{S_2}$ according to the next corollary.

Let 
\begin{equation}
S^{(j)}=\left(
\begin{array}{c c}
A_j & B_j\\
C_j & D_j
\end{array}
\right) \in \operatorname*{Sp}(2n,\mathbb{R})~,
\label{eqHardy62}
\end{equation}
for $j=1,2$, such that:
\begin{equation}
\det\left(B_2 A_1^T-A_2B_1^T \right) \neq 0~.
\label{eqHardy62A}
\end{equation}

The inverse of $S_1$ is given by:
\begin{equation}
S_1^{-1}=JS_1^TJ^T=\left(
\begin{array}{c c}
D_1^T & -B_1^T\\
-C_1^T & A_1^T
\end{array}
\right) \in \operatorname*{Sp}(2n,\mathbb{R})~.
\label{eqHardy63}
\end{equation}
We also define:
\begin{equation}
S=S_2S_1^{-1}=\left(
\begin{array}{c c}
A & B\\
C & D
\end{array}
\right) \in \operatorname*{Sp}(2n,\mathbb{R})~,
\label{eqHardy64}
\end{equation}
where:
\begin{equation}
\left\{
\begin{array}{l l}
A=A_2D_1^T-B_2C_1^T~,~ & B=B_2A_1^T-A_2B_1^T\\
& \\
C=C_2D_1^T-D_2C_1^T~,~ & D= D_2A_1^T-C_2B_1^T
\end{array}
\right.
\label{eqHardy65}
\end{equation}
In view of (\ref{eqHardy62A}), we conclude that $S$ is a free symplectic matrix.

\begin{corollary}\label{CorollaryMultidimensionalHardyLCT}
Let $M,N$ be two real, symmetric, positive-definite $n \times n$ matrices and a function $f : \mathbb{R}^n \to \mathbb{C}$, such that $f \neq 0$ and:
\begin{eqnarray}
|(\widehat{S^{(1)}}f) (x)| \lesssim e^{-M x^2} \label{eqHardy65A}\\
|(\widehat{S^{(2)}}f) (\xi)| \lesssim e^{-N \xi^2} \label{eqHardy65B}
\end{eqnarray}
where $S^{(1)},S^{(2)},S \in \operatorname*{Sp}(2n,\mathbb{R})$ are as above. Moreover, we denote by $\lambda_1, \cdots, \lambda_n$ the eigenvalues of $MB^TNB$.

\begin{enumerate}
\item We must have: $\lambda_j \leq \pi^2$, for all $j=1, \cdots, n$.

\item If $\lambda_1 = \lambda_2=\cdots= \lambda_n= \pi^2$, then there exists a constant $K \in \mathbb{C}$, such that:
\begin{equation}
(\widehat{S^{(1)}}f) (x)= 
Ke^{-x \cdot \left(M+i \pi B^{-1}A\right) x}~.
\label{eqHardy66}
\end{equation}
\end{enumerate}
\end{corollary}

\begin{proof}
The result is an immediate consequence of the previous Theorem if we consider the function $g(x)=(\widehat{S^{(1)}}f) (x)$ and its LCT $\widehat{g}_S(\xi)= (\widehat{S}g)(\xi)= \widehat{f}_{S_2} (\xi)$.
\end{proof}

\begin{remark}\label{RemarkHardyWignerfunction}
Hardy's uncertainty principle can be formulated in terms of Wigner functions (as done in \cite{GoLu}). This formulation is actually stronger than Hardy's uncertainty principle. By the simple linear transformation $D^{(1,2)}$ (see (\ref{eq35})) we can obtain a similar formulation for the bilinear distribution $W_{\vartheta} f$. The result is expressed in terms of the $\vartheta$-Williamson invariants (see \cite{Dias5}).
\end{remark}

\section{Paley-Wiener Theorem for the LCT}

In this section, we state and prove the Paley-Wiener Theorem \cite{Wiener} for the LCT.

\begin{theorem}\label{TheoremPaleyWiener1}
Let $S \in \operatorname*{Sp}(2n,\mathbb{R})$ be of the form (\ref{eqSymp3.1}) with $\det B \neq 0$.

An entire function $g(z)$ on $\mathbb{C}^n$ is the LCT $g= \widehat{f}_{S}$ of a function $f$ with support in the ball $B_R=\left\{x \in \mathbb{R}^n: ~|x| \leq R \right\}$ if and only if, for each $N \in \mathbb{N}$, there is a constant $C_N >0$, such that:
\begin{equation}
|g(z)| \leq C_N \frac{e^{- 2\pi \xi \cdot \left(DB^{-1}\right)_+\eta+ 2 \pi R |B^{-1} \eta|}}{\left(1+ |z|\right)^N}~, 
\label{eqPaleyWiener1}
\end{equation}
for all $z= \xi + i \eta \in \mathbb{C}^n$.
\end{theorem}

\begin{proof}
Suppose that $f(x)=0$ for $|x|>R$ and $f \in C^{\infty} (\mathbb{R}^n)$. Let
\begin{equation}
\widehat{f}_S(z)=\frac{1}{i^{n/2} \sqrt{\det B}} \int_{\mathbb{R}^n} f(x) e^{i \pi z \cdot DB^{-1} z+ i \pi x \cdot B^{-1} A x- 2 i \pi  x \cdot B^{-1} z} dx ~,
\label{eqPaleyWiener2}
\end{equation}
for $z =(z_1, \cdots, z_n) \in \mathbb{C}^n$.

Let us define the "covariant derivatives":
\begin{equation}
\left\{
\begin{array}{l}
\nabla_x = B \left[\partial_x + i \pi \left(B^{-1}A +  (B^{-1}A)^T \right)\cdot x \right]\\
\\
\widetilde{\nabla}_x = B \left[-\partial_x  + i \pi \left(B^{-1}A +  (B^{-1}A)^T \right)\cdot x  \right]
\end{array}
\right.
\label{eqPaleyWiener3}
\end{equation}
where $\partial_x= \left(\frac{\partial}{\partial x_1}, \cdots, \frac{\partial}{\partial x_n}\right)$. 

In the sequel we denote by $\nabla_{x_j}f$ the $j$-th component of $\nabla_x f$, for $j=1, \cdots, n$. We also define $\nabla_x^{\alpha}=\nabla_{x_1}^{\alpha_1} \cdots \nabla_{x_n}^{\alpha_n}$ for a given multi-index $\alpha =(\alpha_1, \cdots \alpha_n)$, $\alpha_j=0,1,2 \cdots$, with the understanding that $ \nabla_{x_j}^{0}f=f$.

From (\ref{eqPaleyWiener2}) we have:
\begin{equation}
\begin{array}{c}
\left|\widehat{f}_S(z)\right|\leq\frac{1}{ \sqrt{|\det B|}} \int_{|x| \leq R} |f(x)| e^{- \pi \xi \cdot DB^{-1} \eta - \pi \eta \cdot DB^{-1} \xi+ 2  \pi  x \cdot B^{-1} \eta} dx\\
\\
\leq \frac{\|f\|_{L^1}}{\sqrt{|\det B|}} e^{-2 \pi \xi \cdot \left(DB^{-1}\right)_+ \eta + 2 \pi R |B^{-1}\eta|}~,
\end{array}
\label{eqPaleyWiener4}
\end{equation}
for all $z= \xi + i \eta \in \mathbb{C}^n$. We may thus safely differentiate under the integral sign and we conclude that $\widehat{f}_S(z)$ is an entire function.

We then have with an integration by parts for $j=1, \cdots, n$: 
\begin{equation}
\begin{array}{c}
\frac{i^{-n/2}}{ \sqrt{\det B}} \int_{|x| \leq R} \left(\nabla_{x_j}^{\alpha_j} f(x) \right) \exp \left(i \pi z \cdot DB^{-1} z + i \pi x \cdot B^{-1} Ax - 2 i \pi x \cdot B^{-1} z \right) dx=\\
\\
= \frac{i^{-n/2}}{ \sqrt{\det B}} \int_{|x| \leq R}  f(x)  \widetilde{\nabla}_{x_j}^{\alpha_j} \exp \left(i \pi z \cdot DB^{-1} z + i \pi x \cdot B^{-1} Ax - 2 i \pi x \cdot B^{-1} z \right) dx=\\
\\
= (2 i \pi z_j)^{\alpha_j} \widehat{f}_S (z)~.
\end{array}
\label{eqPaleyWiener5}
\end{equation}
In particular, for $\alpha_j=N$:
\begin{equation}
\begin{array}{c}
z_j^N \widehat{f}_S (z)= \frac{i^{-n/2}}{(2 \pi i)^N  \sqrt{\det B}} \int_{|x| \leq R} \left(\nabla_{x_j}^{N} f(x) \right) \exp \left(i \pi z \cdot DB^{-1} z + \right.\\
\\
\left. + i \pi x \cdot B^{-1} Ax - 2 i \pi x \cdot B^{-1} z \right) dx
\end{array}
\label{eqPaleyWiener6}
\end{equation}
We conclude that, for $z= \xi +i \eta \in \mathbb{C}^n$:
\begin{equation}
|z_j|^2 | \widehat{f}_S (z)|^{2/N} \leq 
\frac{e^{-\frac{2 \pi}{N} \xi \cdot \left(DB^{-1}\right)_+ \eta +\frac{4 \pi}{N} R |B^{-1} \eta|}}{(2 \pi)^2 |\det B|^{1/N}} \| \nabla_{x_j}^{N} f\|_{L^1 (B_R)}^{2/N}~.
\label{eqPaleyWiener7}
\end{equation}
Summing over $j=1, \cdots, n$ and taking the square root yields:
\begin{equation}
|z|~| \widehat{f}_S (z)|^{1/N} \leq 
\frac{e^{-\frac{\pi}{N} \xi \cdot \left(DB^{-1}\right)_+ \eta +\frac{2 \pi}{N} R |B^{-1} \eta|}}{2 \pi |\det B|^{1/2N}} \sqrt{\sum_{j=1}^n\| \nabla_{x_j}^{N} f\|_{L^1 (B_R)}^{2/N}}~.
\label{eqPaleyWiener8}
\end{equation}
On the other hand, from (\ref{eqPaleyWiener4}) we have:
\begin{equation}
| \widehat{f}_S (z)|^{1/N} \leq 
\frac{e^{-\frac{\pi}{N} \xi \cdot \left(DB^{-1}\right)_+ \eta +\frac{2 \pi}{N} R |B^{-1} \eta|}}{|\det B|^{1/2N}} \| f\|_{L^1 (B_R)}^{1/N}~.
\label{eqPaleyWiener9}
\end{equation}
From (\ref{eqPaleyWiener8}) and (\ref{eqPaleyWiener9}), we obtain:
\begin{equation}
(1+|z|)^N| \widehat{f}_S (z)| \leq 
\frac{e^{-\pi \xi \cdot \left(DB^{-1}\right)_+ \eta +2 \pi R |B^{-1} \eta|}}{\sqrt{|\det B|}} \left[    \| f\|_{L^1 (B_R)}^{1/N} +\frac{1}{2 \pi} \sqrt{ \sum_{j=1}^n\|\nabla_{x_j}^{N} f\|_{L^1 (B_R)}^{2/N}} \right]~.
\label{eqPaleyWiener10}
\end{equation}
It follows that:
\begin{equation}
| \widehat{f}_S (z)| \leq C_N \frac{e^{-\pi \xi \cdot \left(DB^{-1}\right)_+ \eta +2 \pi R |B^{-1} \eta|}}{(1+|z|)^N}~,
\label{eqPaleyWiener11}
\end{equation}
where the constant $C_N$ depends on $f$, $N$, $R$, $n$ and on the matrices $A$ and $B$, but is independent of $z$.

Conversely, suppose that $g(z)$ is an entire function of $z=\xi+ i \eta \in \mathbb{C}^n$, such that (\ref{eqPaleyWiener1}) holds for all $z$.

We claim that, for each fixed $\eta \in \mathbb{R}^n$, the function
\begin{equation}
h_{\eta} (\xi)=g(\xi + i \eta) e^{- i \pi (\xi+ i \eta) (DB^{-1})^T (\xi + i \eta)}= g(z) e^{- i \pi z (DB^{-1})^T z}
\label{eqPaleyWiener12}
\end{equation}
of $\xi \in \mathbb{R}^n$ is in $\mathcal{S} (\mathbb{R}^n)$. Let us prove this claim.

For any $\alpha = (\alpha_1, \cdots, \alpha_n),~ \beta = (\beta_1, \cdots, \beta_n) \in \mathbb{N}_0^n$, we have by Cauchy's formula:
\begin{equation}
\begin{array}{c}
\partial_z^{\alpha} \left(g(z) e^{- i \pi z \cdot(DB^{-1})^Tz} \right)=\\
\\
= \frac{\alpha !}{(2 i \pi)^{|\alpha|}} \oint_{\gamma_1} \cdots \oint_{\gamma_n} \frac{g(z^{\prime})}{(z_1^{\prime}-z_1)^{\alpha_1+1} \cdots (z_n^{\prime}-z_n)^{\alpha_n+1}} d z_n^{\prime} \cdots d z_1^{\prime}~,
\end{array}
\label{eqPaleyWiener13}
\end{equation}
where as usual $\partial_z^{\alpha} =\partial_{z_1}^{\alpha_1} \cdots \partial_{z_n}^{\alpha_n} $, and $\gamma_1, \cdots, \gamma_n$ are certain contours enclosing $z_1, \cdots, z_n$, respectively. 

Without loss of generality we may assume that $\gamma_j$ is a circle of radius $\frac{1}{\sqrt{n}}$ centered at $z_j$, so that $|z_j^{\prime}-z_j|= \frac{1}{\sqrt{n}}$ and $|z^{\prime}-z|=1$.

From (\ref{eqPaleyWiener1}), we then get:
\begin{equation}
\begin{array}{c}
  \left|\xi^{\beta} \partial_z^{\alpha} \left(g(z) e^{-i \pi z \cdot (DB^{-1})^Tz}  \right)\right| \leq \\
 \\
 \leq \frac{\alpha! C_N}{(2\pi)^{|\alpha|}} \oint_{\gamma_1} \cdots \oint_{\gamma_n} \frac{|\xi|^{\beta} e^{2 \pi R |B^{-1} \eta^{\prime}|}}{|z_1^{\prime}-z_1|^{\alpha_1 +1} \cdots |z_n^{\prime}-z_n|^{\alpha_n +1} (1+ |z^{\prime}|)^N} d z_1^{\prime} \cdots d z_n^{\prime}~, \end{array}
 \label{eqPaleyWiener14}
\end{equation}
where we set $z^{\prime}=(z_1^{\prime}, \cdots, z_n^{\prime}) \in \mathbb{C}^n$, $z_j^{\prime}=\xi_j^{\prime}+i \eta_j^{\prime}$.

Let us change variables to $\zeta_j= z_j^{\prime}-z_j$, where $|\zeta_j|=\frac{1}{\sqrt{n}}$. We obtain:
\begin{equation}
\begin{array}{c}
  \left|\xi^{\beta} \partial_z^{\alpha} g(z) e^{-i \pi z \cdot (DB^{-1})^Tz}  \right| \leq \\
 \\
 \leq \frac{\alpha! C_N (\sqrt{n})^{|\alpha|+n}}{(2\pi)^{|\alpha|}} \oint_{C_{1/\sqrt{n}}} \cdots \oint_{C_{1/\sqrt{n}}} \frac{|\xi|^{\beta} e^{2 \pi R \left( |B^{-1} \eta| + \|B^{-1}\|_{op}\right)}}{ (1+ |z+ \zeta|)^N} d \zeta_1 \cdots d \zeta_n~, 
\end{array}
 \label{eqPaleyWiener15}
\end{equation}
where $C_{1/\sqrt{n}}$ is the circle of radius $\frac{1}{\sqrt{n}}$ centered at the origin and $\|\cdot \|_{op}$ denotes the Frobenius norm of a matrix.

On the other hand, if we set $\zeta=\zeta_R+i \zeta_I$, $\zeta_R,\zeta_I \in \mathbb{R}^n$:
\begin{equation}
|\xi|=|\xi+\zeta_R-\zeta_R| \leq |\xi+\zeta_R|+|\zeta_R| \leq |z+\zeta|+|\zeta|   =|z+\zeta|+1~.
 \label{eqPaleyWiener16}
\end{equation}
Thus:
\begin{equation}
|\xi|^{|\beta|} \leq (1+ |z+ \zeta|)^{|\beta|}~.
 \label{eqPaleyWiener17}
\end{equation}
From (\ref{eqPaleyWiener15},\ref{eqPaleyWiener17}):
\begin{equation}
\begin{array}{c}
  \left|\xi^{\beta} \partial_z^{\alpha} \left(g(z) e^{-i \pi z \cdot (DB^{-1})^Tz} \right) \right| \leq \\
 \\
 \leq \frac{\alpha! C_N (\sqrt{n})^{|\alpha|+n}}{(2\pi)^{|\alpha|}}  e^{2 \pi R \left( |B^{-1} \eta| + \|B^{-1}\|_{op}\right)} \times \\
 \\
 \times \oint_{C_{1/\sqrt{n}}} \cdots \oint_{C_{1/\sqrt{n}}} \frac{1}{ (1+ |z+ \zeta|)^{N-|\beta|}} d \zeta_1 \cdots d \zeta_n~, 
\end{array}
 \label{eqPaleyWiener18}
\end{equation}
Since $N \in \mathbb{N}_0$ is arbitrary, we may choose it sufficiently large so that $N \geq |\beta|$. It follows that:
\begin{equation}
\begin{array}{c}
\oint_{C_{1/\sqrt{n}}} \cdots \oint_{C_{1/\sqrt{n}}} \frac{1}{ (1+ |z+ \zeta|)^{N-|\beta|}} d \zeta_1 \cdots d \zeta_n \leq \\
\\
\leq \oint_{C_{1/\sqrt{n}}} \cdots \oint_{C_{1/\sqrt{n}}} 1 d \zeta_1 \cdots d \zeta_n = 
\left(\frac{2 \pi}{\sqrt{n}}\right)^n~,
\end{array}
\label{eqPaleyWiener19}
\end{equation}
and:
\begin{equation}
  \left|\xi^{\beta} \partial_z^{\alpha} \left(g(z) e^{-i \pi z \cdot (DB^{-1})^Tz} \right) \right| \leq  \frac{\alpha! C_N (\sqrt{n})^{|\alpha|}}{(2\pi)^{|\alpha|-n}}  e^{2 \pi R \left( |B^{-1} \eta| + \|B^{-1}\|_{op}\right)} ~, 
 \label{eqPaleyWiener20}
\end{equation}
which proves that $h_{\eta} (\xi) \in \mathcal{S} (R^n)$, as claimed.

Let 
\begin{equation}
\begin{array}{c}
f(x)= \left(\mathcal{L}_{S^{-1}} g \right) (x)=\\
\\
=\frac{i^{-n/2}}{\sqrt{\det B}} \int_{\mathbb{R}^n} g(\xi) e^{- i \pi x \cdot (B^{-1}A)^T x - i  \pi \xi \cdot (DB^{-1})^T \xi +  2 i \pi \xi \cdot (B^{-1})^T x} d \xi ~.
\end{array}
\label{eqPaleyWiener21}
\end{equation}
Notice that we may rewrite $f(x)$ as:
\begin{equation}
\begin{array}{c}
f(x) =\frac{e^{- i \pi x \cdot (B^{-1}A)^T x}}{i^{n/2}\sqrt{\det B}} \int_{\mathbb{R}^n} h_0 (\xi) e^{  2 i \pi \xi \cdot (B^{-1})^T x} d \xi =\\
\\
= \frac{e^{- i \pi x \cdot (B^{-1}A)^T x}}{i^{n/2}\sqrt{\det B}} \widehat{h}_0\left(- (B^{-1})^T x \right) ~,
\end{array}
\label{eqPaleyWiener22}
\end{equation}
where $h_0(\xi)$ is given by (\ref{eqPaleyWiener12}) for fixed $\eta=0$.

Since the Fourier transform is an isomorphism in $\mathcal{S} (\mathbb{R}^n)$, we conclude that $f \in \mathcal{S} (\mathbb{R}^n)$ and from (\ref{eqPaleyWiener21}) it follows that $g= \widehat{f}_S$. It remains to prove that the support of $f$ belongs to the ball of radius $R$.

Because of estimates (\ref{eqPaleyWiener11}) and Cauchy's theorem, we can shift the region of integration in (\ref{eqPaleyWiener21}) and obtain:
\begin{equation}
\begin{array}{c}
f(x)= \frac{i^{-n/2}}{ \sqrt{\det B}} \int_{\mathbb{R}^n} g(\xi + i \eta) \times \\
\\
\times e^{-i \pi x \cdot (B^{-1}A)^T x - i \pi (\xi+i \eta) (DB^{-1})^T(\xi+ i \eta) +2 i \pi (\xi+i \eta) (B^T)^{-1} x} d \xi ~. 
\end{array}
\label{eqPaleyWiener23}
\end{equation}
Consequently from (\ref{eqPaleyWiener11}):
\begin{equation}
\begin{array}{c}
|f(x)| \leq \frac{C_N e^{2 \pi R |B^{-1} \eta|-2 \pi \eta \cdot (B^{-1})^T x}}{ \sqrt{|\det B|}} \int_{\mathbb{R}^n} \frac{1}{(1+ |\xi + i \eta|)^N}d \xi\\
\\
\leq  \frac{C_N e^{2 \pi R |B^{-1} \eta|-2 \pi \eta \cdot (B^{-1})^T x}}{ \sqrt{|\det B|}} \int_{\mathbb{R}^n} \frac{1}{(1+ |\xi |)^N}d \xi~. 
\end{array}
\label{eqPaleyWiener24}
\end{equation}
We may choose $N$ large enough so that the integral on the right-hand side is convergent and we get:
\begin{equation}
|f(x)| \leq D_N e^{2 \pi R |B^{-1} \eta|-2 \pi \eta \cdot (B^{-1})^T x}~,
\label{eqPaleyWiener25}
\end{equation}
for some contant $D_N>0$.

Let $\sigma=B^{-1}\eta $. Since $\eta$ (and hence $\sigma$) is arbitrary, we may choose it so that $\sigma \neq 0 $ points in the same direction as $x$: $\sigma \cdot x= |\sigma|~|x|$.

Thus:
\begin{equation}
2 \pi R |B^{-1} \eta|-2 \pi \eta \cdot (B^{-1})^T x=
2 \pi R |\sigma|-2 \pi \sigma \cdot x = 2 \pi |\sigma| \left(R- |x|\right)~.
\label{eqPaleyWiener26}
\end{equation}
If we send $|\sigma| \to \infty$, we conclude from (\ref{eqPaleyWiener25}) that $f(x)$ vanishes for $|x| >R$.
\end{proof}

\begin{remark}\label{RemarkPaleyWiener}
Suppose that an entire function $g(z)$ is such that for each $N \in \mathbb{N}$, there is a constant $C_N >0$, such that:
\begin{equation}
|g(z)| \leq C_N \frac{e^{ 2 \pi R | \eta|}}{\left(1+ |z|\right)^N}~, 
\label{eqPaleyWiener27}
\end{equation}
for all $z= \xi + i \eta \in \mathbb{C}^n$.

Then, by the Paley-Wiener theorem, there exists a function $f$ with support contained in the ball $B_R$ of radius $R$, such that $g(\xi)= \widehat{f} (\xi)$.

Alternatively, suppose that there exists some $N=N_0 \in \mathbb{N}$ such that (\ref{eqPaleyWiener27}) does not hold, but (\ref{eqPaleyWiener1}) is valid for some matrices $B,D$. Then Theorem \ref{TheoremPaleyWiener1} implies that there exists a function $f$ such that $g(\xi)= \widehat{f}_S (\xi)$ with support contained in the ball $B_R$. 

We can understand this in the following geometric terms. Let us consider the $n=1$ case for simplicity. Suppose that $f$ is a function with support $\text{Supp} (f) \subset \left[-R,R \right]$. Then the associated Wigner function $W_{\sigma}f$ has support contained in the infinite strip $\left[-R,R \right] \times \mathbb{R}$. The Fourier transform $\widehat{f}$ extends to an entire function on $\mathbb{C}^2$ which satisfies (\ref{eqPaleyWiener27}).

Now suppose that we perform some LCT $\widehat{S}_{\theta}$ which amounts to a counter clockwise rotation in the time-frequency plane by an angle $0 < \theta <\frac{\pi}{2}$. The support of the corresponding Wigner distribution $W_{\sigma}\left( \widehat{S}_{\theta}f\right)(z)=W_{\sigma}f(S_{\theta}^{-1} z) $ is now contained between the two straight lines $x \cos \theta + \xi \sin \theta = \pm R$. Moreover the LCT $\left(\widehat{S}_{\theta} f \right) (\xi)$ extends to an entire function on $\mathbb{C}^2$ which satisfies an estimate of the form (\ref{eqPaleyWiener1}), but not (\ref{eqPaleyWiener27}).  
\end{remark} 

\section*{Acknowledgements}

M. de Gosson has been funded by the grant P 33447 of the Austrian Research Foundation FWF.

\pagebreak

****************************************************************

\textbf{Author's addresses:}

\begin{itemize}
\item \textbf{Nuno Costa Dias and Jo\~ao Nuno Prata: }Escola Superior N\'autica Infante D. Henrique. Av.
Eng. Bonneville Franco, 2770-058 Pa\c{c}o d'Arcos, Portugal and Grupo de F\'{\i}sica
Matem\'{a}tica, Departamento de Matem\'atica, Faculdade de Cí\^encias, Universidade de Lisboa, Campo Grande, Edif\'{\i}cio C6, 1749-016 Lisboa, Portugal

\item \textbf{Maurice A. de Gosson:} Universit\"{a}t Wien, Fakult\"{a}t
f\"{u}r Mathematik--NuHAG, Nordbergstrasse 15, 1090 Vienna, Austria
\end{itemize}

 \end{document}